\documentclass[10pt]{article}

\usepackage{amssymb,amsmath,amsthm}
\usepackage{amsfonts}
\usepackage{graphicx}
\usepackage[usenames]{color}
\usepackage{url}
\usepackage[colorlinks,linktocpage,linkcolor=blue]{hyperref}
\usepackage{algorithm,algorithmic}
\usepackage{dsfont,verbatim}
\usepackage{eepic,epic}
\usepackage{epsfig,subfigure,epstopdf}
\graphicspath{{./pics/}}
\allowdisplaybreaks

\newtheorem{theorem}{Theorem}[section]
\newtheorem{lemma}{Lemma}[section]

\newtheorem{example}{Example}[section]
\newtheorem{proposition}{Proposition}[section]

\newtheorem{remark}{Remark}[section]

\numberwithin{equation}{section}

\newcommand{\al}{\alpha}
\newcommand{\fy}{\varphi}

\def\Dal{{\partial_t^\al}}

\def\Om{\Omega}
\def\II{(\Om)}

\def\bPtau{\bar\partial_\tau}

\newcount\icount

\title{An Analysis of the Crank-Nicolson Method for Subdiffusion} 
\author{Bangti Jin\thanks{Department of Computer Science, University College London, Gower Street, London, WC1E 2BT, UK
(\texttt{b.jin@ucl.ac.uk, bangti.jin@gmail.com})}
\and Buyang Li\thanks{ Department of Applied Mathematics, The Hong Kong Polytechnic University, Kowloon, Hong Kong.
(\texttt{libuyang@gmail.com})}
\and Zhi Zhou\thanks{Department of Applied Physics and Applied Mathematics,
Columbia University, 500 W. 120th Street, New York, NY 10027, USA (\texttt{zhizhou0125@gmail.com})}}
\date{\today}

\begin{document}
\maketitle
\begin{abstract}
In this work, we  analyze a Crank-Nicolson type time stepping scheme for
the subdiffusion equation, which involves a Caputo fractional derivative of order $\alpha\in (0,1)$
in time. It hybridizes the backward Euler convolution quadrature with a $\theta$-type method, with
the parameter $\theta$ dependent on the fractional order $\alpha$ by $\theta=\alpha/2$, and
naturally generalizes the classical Crank-Nicolson method. We develop essential initial corrections
at the starting two steps for the Crank-Nicolson scheme, and together with the Galerkin finite element method in space,
obtain a fully discrete scheme. The overall scheme is easy to implement, and robust with respect to data
regularity. A complete error analysis of the fully discrete scheme is provided,
and a second-order accuracy in time is established for both smooth and nonsmooth problem data.
Extensive numerical experiments are provided to illustrate its accuracy, efficiency and robustness,
and a comparative study also indicates its competitive with existing schemes.
\end{abstract}

\section{Introduction}\label{sec:intro}
Let $\Omega$ be a bounded convex polygonal domain in $\mathbb R^d\,(d=1,2,3)$
with a boundary $\partial\Omega$, and $T>0$ be a fixed value.
We are interested in efficient numerical methods for
the following fractional-order evolution equation of
$u(t):(0,T)\rightarrow H^1_0(\Omega)\cap H^2(\Omega)$:
\begin{equation}\label{eqn:fde}
   \Dal u(t) -\Delta u(t)= f(t)\quad \text{for} \,\,\,t\in (0,T),\\
\end{equation}
where $\Delta:H^1_0(\Omega)\cap H^2(\Omega)\rightarrow L^2(\Omega)$ denotes the Laplacian,
$f:(0,T)\rightarrow L^2(\Omega)$ is a given function, and the notation $\Dal u$, $0<\al<1$,
denotes the Caputo fractional derivative of order $\al$ with respect to $t$, defined
by \cite[pp. 91]{KilbasSrivastavaTrujillo:2006}
\begin{equation}\label{McT}
   \Dal u(t):= \frac{1}{\Gamma(1-\al)} \int_0^t(t-s)^{-\al}\frac{d}{ds}u(s)\, ds,
\end{equation}
with $\Gamma(\cdot)$ being the Gamma function defined by
$\Gamma(s):=\int_0^\infty t^{s-1}e^{-t}dt$ for $\Re(s)>0$.
The model \eqref{eqn:fde} is {subject to} a zero boundary condition
$u =0$ on $\partial\Omega\times(0,T]$, and the following initial condition
\begin{equation*}
      u(0)=v,\quad\text{in}\,\,\,\Omega,
\end{equation*}
where $v$ is a given function defined on the domain $\Omega$.
The model \eqref{eqn:fde} with $0<\alpha<1$ is popular for modeling subdiffusion processes, in which the
mean-squared displacement of  particle motion grows only sublinearly with the time $t$, instead
of the linear growth for normal diffusion. It has been applied in several
fields, e.g., solute transport in heterogeneous media and cytoplasmic crowding
in living cells; see \cite{Metzler:2014} for an extensive list.

Recently, there has been much interest in developing efficient numerical schemes for \eqref{eqn:fde}.
A number of time stepping schemes have been proposed, which roughly can be divided into two groups,
i.e., L1 type scheme and convolution quadrature (CQ). L1 type schemes are of finite difference
nature, and can be derived by polynomial interpolation \cite{Diethelm:1997,LinXu:2007,
GaoSunZhang:2014,Alikhanov:2015,CaoZhangKarniadakis:2015,LvXu:2016}.
These schemes were derived under the assumption that the solution $u$ is smooth, and
require high solution regularity for error estimates. See also \cite{McLeanMustapha:2009,MustaphaAbdallahFurati:2014} for discontinuous Galerkin methods. CQ due to \cite{Lubich:1986,Lubich:1988}
presents a flexible framework for devising high-order time stepping schemes for \eqref{eqn:fde},
and merits excellent stability property. Thus
it has been customarily applied \cite{Yuste:2006,CuestaLubichPalencia:2006,JinLazarovZhou:SISC2016,Dimitrov:2014,ZengLiLiuTurner:2015}. For both L1 type and CQ schemes, 
proper corrections are necessary in order to obtain high-order convergence for
general problem data, including very smooth data. However, to the best of our limited knowledge, for problem \eqref{eqn:fde}, so far
this has been only done in \cite{CuestaLubichPalencia:2006} and \cite{JinLazarovZhou:SISC2016} for
CQ generated by the second-order BDF. Hence, it remains imperative to develop and analyze high-order schemes
robust with respect to data regularity.

In this paper, we present an analysis of a robust $O(\tau^2)$ accurate fractional Crank-Nicolson scheme, with finite element
space discretization. Let $\tau$ be the constant time step size and $h$ the mesh size. Using the
time-stepping scheme developed in \cite{Dimitrov:2014} and standard Galerkin finite element method in space,
we propose a fully discrete scheme approximates the solution $u(t_n)$ by $U_h^n$, $n=1,2,\ldots,N$:
\begin{equation}\label{eqn:CN}
    \bar\partial_\tau^\alpha (U_h^n-v_h) - \big(1-\tfrac{\alpha}{2}\big)\Delta_h U_h^{n}
    -\tfrac{\alpha}{2}\Delta_h U_h^{n-1}=\big(1-\tfrac{\alpha}{2}\big)F_h^{n}+\tfrac{\alpha}{2}F_h^{n-1},
\end{equation}
where $\Delta_h:X_h\rightarrow X_h$ denotes the Galerkin approximation of the Laplacian on
a finite element subspace $X_h\subset H^1_0(\Omega)$, $F_h^n=P_hf(t_n)$ denotes the $L^2$-projection of
$f(t_n)$ onto $X_h$, and $v_h\in X_h$ is an approximation to the initial data $v$. In \eqref{eqn:CN},
$\bar\partial_\tau^\alpha \varphi^n$ denotes the backward Euler CQ approximation to the Riemann-Liouville
fractional derivative $^R\kern-.2mm\partial_t^\alpha \varphi(t_n)$ (cf. \eqref{eqn:RL} below) defined by:
\begin{equation}\label{eqn:BE}
  \bar\partial_\tau^\alpha \varphi^n := \tau^{-\alpha}\sum_{j=0}^nb_{n-j}\varphi^j,\quad \mbox{with}\quad
\sum_{j=0}^\infty b_j\xi^j:=(1-\xi)^\alpha,
\end{equation}
where the weights $b_j$ are available in closed form: $b_j=(-1)^j\frac{\Gamma(\alpha+1)}{\Gamma(j+1)\Gamma(\alpha-j+1)}$.
Clearly, for $\alpha=1$, the scheme \eqref{eqn:CN} recovers the classical Crank--Nicolson method \cite{CrankNicolson:1947}, and
thus it represents a natural extension of the latter to the fractional case, which has long been missing in the literature.
For $\alpha\in(0,1)$, the scheme \eqref{eqn:CN} hybridizes the backward Euler CQ with the $\theta$-type method
with a weight $\theta=\alpha/2$. This choice was motivated by the fact that it yields a
local truncation error $O(\tau^2)$ under certain compatibility conditions; see
Section \ref{subsec:fully} for details.
The numerical experiments therein show that it is indeed second-order accurate in time
if the solution $u$ is sufficiently smooth.

However, the solution $u$ of problem \eqref{eqn:fde} can be weakly singular near $t=0$ even for smooth
problem data \cite{SakamotoYamamoto:2011,Stynes:2016}, and thus a straightforward implementation of
\eqref{eqn:CN} yields only an $O(\tau)$ convergence, cf. Table \ref{tab:1st}, as for other
high-order time stepping schemes. 
Inspired by \cite{CuestaLubichPalencia:2006}, we shall correct it at the starting two steps, leading to
a novel corrected scheme, cf. \eqref{eqn:fully-mod} below. The new scheme has two
distinct features, which make it very attractive. (i) Since it employs the backward Euler CQ and
changes only the first two steps, it is
straightforward to implement. (ii) It is robust with respect to data regularity: it can achieve an $O(\tau^2)$
convergence in time for nonsmooth initial data $v$ and source term $f$ incompatible with
$v$ at $t=0$ (cf. Theorems \ref{thm:v-nonsmooth} and \ref{thm:inhomog-fully}).
Our numerical experiments in Section \ref{sec:numerics} fully confirm its accuracy and robustness.

The contributions of the work are threefold. First, we develop essential initial corrections
for the scheme \eqref{eqn:CN} in order to restore the $O(\tau^2)$ accuracy for nonsmooth data.
It presents a new robust second-order scheme for \eqref{eqn:fde}, competitive with
the corrected second-order BDF.
Second, we provide a complete convergence analysis of the corrected scheme under realistic
regularity conditions on the data. For example, for $v\in L^2(\Omega)$ and $f\equiv0$, we show in Theorem
\ref{thm:v-nonsmooth} the following error estimate
\begin{equation*}
   \| u_h(t_n)-U_h^n \|_{L^2(\Omega)} \le c\tau^2 t_n^{-2} \| v \|_{L^2(\Omega)}, \quad n \ge 1,
\end{equation*}
where $u_h$ is the semidiscrete Galerkin solution, cf. \eqref{eqn:fem}.
Some preliminary analysis of the scheme \eqref{eqn:CN} was given in \cite[Theorem 1]{Dimitrov:2014} under
high regularity assumption on the solution, i.e., $u\in C^4[0,T]$, and restrictive compatibility condition
$u^{(i)}(0)=0$, $i=0,1,2$. We shall derive
optimal error estimates that are directly expressed in terms of data regularity. As a by-product,
we also give the guideline for constructing initial corrections for other schemes, cf. Remark \ref{rem:correct}. Third, the proof relies
on the discrete Laplace transform and a refined analysis of the kernel function,
largely inspired by a strategy outlined in \cite{LubichSloanThomee:1996}. Due to the hybridization
of the $\theta$ method with the backward Euler CQ, the scheme lacks a simple convolution structure,
leading to a complex kernel, and
is challenging to analyze. We shall develop a general strategy in Lemma \ref{lem:est-hom-2} to
overcome the challenge. Thus the convergence analysis
differs substantially from existing works \cite{CuestaLubichPalencia:2006,JinLazarovZhou:SISC2016},
for which the requisite basic estimates on the kernel are well known.

The rest of the paper is organized as follows. In Section \ref{sec:CN} we rederive
the scheme \eqref{eqn:CN} and develop the initial corrections. Then in Section \ref{sec:conv},
we present a complete convergence analysis of the corrected scheme. The focus is
on the time discretization error, since the error analysis of the semidiscrete Galerkin scheme
is well understood. Last, in Section \ref{sec:numerics} we present extensive
numerical experiments to confirm the convergence rates for both smooth and nonsmooth problem data,
where a comparative study with CQ generated by the second-order BDF, cf.
\cite{JinLazarovZhou:SISC2016}, and the L1-2 scheme, cf. \cite{GaoSunZhang:2014}, also shows clearly
its competitiveness. Throughout, the notation $c$, with or without a subscript,
denotes a generic constant which may differ at different occurrences, but it is
always independent of the mesh size $h$ and time step size $\tau$.

\section{The fractional Crank-Nicolson scheme}\label{sec:CN}
In this part, we derive a fully discrete scheme for problem \eqref{eqn:fde} using a
standard Galerkin FEM in space and the fractional Crank-Nicolson approximation in time.

\subsection{Semidiscrete Galerkin scheme}
Let $\mathcal{T}_h$ be a shape regular, quasi-uniform triangulation of the
domain $\Omega $ into $d$-simplexes, denoted by $T$, with a mesh size $h$. Then
over $\mathcal{T}_h$, we define a continuous piecewise linear finite
element space $X_h$ by
\begin{equation*}
  X_h= \left\{v_h\in H_0^1(\Omega):\ v_h|_T \mbox{ is a linear function},\ \forall\, T \in \mathcal{T}_h\right\}.
\end{equation*}
We define the $L^2(\Omega)$-projection $P_h:L^2(\Omega)\to X_h$ and
the Ritz projection $R_h:H_0^1(\Omega)\to X_h$ by
\begin{equation*}
   \begin{aligned}
     (P_h \fy,\chi_h) &=(\fy,\chi_h) , &&\forall\, \chi_h\in X_h,\\
     (\nabla R_h\fy,\nabla\chi_h) & = (\nabla \fy,\nabla \chi_h),&& \forall\, \chi_h\in X_h,
   \end{aligned}
\end{equation*}
respectively, where $(\cdot,\cdot)$ denotes the inner product of $L^2(\Omega)$. Then the spatially
semidiscrete Galerkin FEM scheme for problem \eqref{eqn:fde} reads: find $u_h(t)\in X_h$ such that
\begin{equation}\label{eqn:fem}
  (\Dal u_h,\chi_h) + (\nabla u_h,\nabla \chi_h) = (f,\chi_h),\quad\forall\, \chi_h\in X_h,
\end{equation}
with the initial condition $u_h(0)=v_h\in X_h$.
The choice $v_h$ depends on the smoothness of the initial data $v$ (\cite{Thomee:2006}):
for $v\in D(\Delta)=H_0^1(\Omega)\cap H^2(\Omega)$, we take $v_h=R_h v $,
and for $v\in L^2(\Omega)$, we take $v_h=P_hv $.
Upon introducing the discrete Laplacian $\Delta_h: X_h\to X_h$, defined by
$-(\Delta_h\fy_h,\chi_h)=(\nabla\fy_h,\nabla\chi_h)$ for all $\fy_h, \chi_h\in X_h,$
we can rewrite \eqref{eqn:fem} as: with $u_h(0)=v_h\in X_h$ and $f_h(t)=P_hf(t)$
\begin{equation}\label{eqn:fdesemidis}
  \Dal u_h(t) -\Delta_h u_h(t) = f_h(t) , \quad \forall\, t>0.
\end{equation}

The semidiscrete scheme \eqref{eqn:fdesemidis} has been analyzed in
\cite{JinLazarovZhou:SIAM2013,JinLazarovPasciakZhou:IMA2014,LeMcLeanLamichhane:2016,KaraaMustapha:2016}, and we refer interested readers to these works for detailed error estimates.

\subsection{Formal derivation of the fractional Crank-Nicolson scheme}\label{subsec:fully}

In this part, we formally derive the fractional Crank-Nicolson scheme \eqref{eqn:CN}. Upon recalling the
defining relation of the Caputo derivative $\partial_t^\alpha \varphi$
in terms of the Riemann-Liouville one \cite[pp. 91, eq. (2.4.1)]{KilbasSrivastavaTrujillo:2006}, i.e.,
$\partial_t^\alpha\varphi(t)={^R\kern-.2mm \partial_t^\alpha}(\varphi(t)-\varphi(0))$, where the
Riemann-Liouville fractional derivative $^R\kern-.2mm\Dal \varphi(t)$, for $0<\alpha<1$,
is defined by: 
\begin{equation}\label{eqn:RL}
  ^R\kern-.2mm\Dal \varphi(t) = \frac{1}{\Gamma(1-\alpha)}\frac{d}{dt}\int_0^t(t-s)^{-\alpha}\varphi(s)ds.
\end{equation}
we can rewrite the
scheme \eqref{eqn:fdesemidis} into
\begin{equation*}
    ^R\kern-.2mm\Dal (u_h(t)-v_h) -\Delta_h u_h(t) = f_h(t),
\end{equation*}

Now consider a uniform partition of $[0,T]$ with time
step size $\tau = T/N$, $N\in\mathbb{N}$, so that $0=t_0<t_1<\ldots<t_N=T$, and $t_n=n
\tau$, $n=0,\ldots,N$. The Riemann-Liouville derivative $^R\kern-.2mm\Dal \varphi(t_n)$ can be
discretized using the backward Euler CQ \eqref{eqn:BE}, and it is $O(\tau)$
accurate \cite[pp. 204-208]{Podlubny:1988}.
In order to achieve an $O(\tau^2)$ accuracy, we aim at deriving a $\theta$-type method with a
suitable weight $\theta$ by Fourier transform. We denote by ${\mathcal F}_t$  the Fourier
transform in  $t$ and by ${\mathcal F}_\xi^{-1}$ the inverse Fourier
transform in $\xi$. Assuming that the function $\varphi$ is smooth over the domain
${\mathbb R}$ and $\varphi=0$ for $t\le 0$, then the function
$$
\bar\partial_\tau^\alpha \varphi(t):=\tau^{-\alpha}\sum_{j=0}^\infty  b_j\varphi(t-j\tau)
$$
coincides with the scheme \eqref{eqn:BE} at $t=t_n$ and satisfies
\begin{align*}
\begin{aligned}
{\mathcal F}_t [\bar\partial_\tau^\alpha \varphi(t)](\xi) & =
\int_{{\mathbb R}}\bar\partial_\tau^\alpha \varphi(t)e^{-{\rm i}t\xi} d t
=\tau^{-\alpha}\sum_{j=0}^\infty \int_{{\mathbb R}} b_j\varphi(t-j\tau)e^{-{\rm i}t\xi} d t  \\
&=\tau^{-\alpha}(1-e^{-{\rm i}\tau\xi})^\alpha {\mathcal F} \varphi(\xi)
=({\rm i}\xi)^\alpha(1-\tfrac{{\rm i}\alpha}{2}\tau\xi+O(\tau^2\xi^2))
{\mathcal F} \varphi(\xi) .
\end{aligned}
\end{align*}
In view of the identity
\begin{align*}
&{\mathcal F}_t [{^R\kern-.2mm\partial_t^\alpha} \varphi(t-s)](\xi)
=({\rm i}\xi)^\alpha e^{-{\rm i}s\xi} {\mathcal F} \varphi(\xi)  =
({\rm i}\xi)^\alpha\big(1- {\rm i}s \xi+O({s^2\xi^2})\big)
 {\mathcal F} \varphi(\xi) ,
\end{align*}
and by the choice $s=\alpha\tau/2$, formally we derive
\begin{align*}
\begin{aligned}
\bar\partial_\tau^\alpha \varphi(t)&={^R\kern-.2mm\partial_t^\alpha} \varphi(t-\alpha\tau/2)
+{\mathcal F}_\xi^{-1}[O(\tau^2\xi^2) ({\rm i}\xi)^\alpha  {\mathcal F} \varphi(\xi) ] \\
&={^R\kern-.2mm}\partial_t^\alpha \varphi(t-\alpha\tau/2)+O(\tau^2) \\
&=(1-\tfrac{\alpha}{2}){^R\kern-.2mm}\partial_t^\alpha \varphi(t)
+ \tfrac{\alpha}{2}{^R\kern-.2mm}\partial_t^\alpha \varphi(t-\tau)+O(\tau^2).
\end{aligned}
\end{align*}
By choosing $t=t_n$ in the preceding expression, it intuitively motivates the scheme \eqref{eqn:CN}.

\begin{remark}
The zero extension to $t<0$ in the formal derivation implicitly imposes certain compatibility conditions at
$t=0$, i.e., $u^{(i)}=0$, $i=0,1,2$.
The estimate $\bar\partial_\tau^\alpha \varphi(t) = {^R\kern-.2mm}\partial_t^\alpha \varphi(t-\frac{\alpha}{2}\tau)+O(\tau^2)$
was first observed in \cite[Theorem 1]{Dimitrov:2014}. It implies that despite the $O(\tau)$ convergence at the node $t_n$, the
approximation $\bar\partial_\tau^\alpha \varphi^n$ is $O(\tau^2)$ accurate at the point $t=t_n-\frac{\alpha}{2}\tau$.
\end{remark}

We illustrate the scheme \eqref{eqn:CN} with one-dimensional numerical examples.
\begin{example}\label{exam:1d}
Consider problem \eqref{eqn:fde} on the unit interval $\Omega=(0,1)$ with $T=1$.
\begin{itemize}
  \item[$\mathrm{(a)}$]$v=0$, and $ f=2t^{2-\al}x(1-x)/\Gamma(3-\al)+2t^2 $. The exact solution $u(x,t)=t^2 x(1-x)$ is smooth.
  \item[$\mathrm{(b)}$] $v=x(1-x)$, and $f=0$.
\end{itemize}
The mesh size $h$ is fixed at $h = 10^{-4}$ so
that the error incurred by spatial discretization is negligible.
\end{example}

Since the exact solution $u$ is smooth and satisfies the compatibility condition in case \ref{exam:1d}(a), the scheme \eqref{eqn:CN}
exhibits an $O(\tau^2)$ rate as expected, cf. Table \ref{tab:smooth_sol},
where the $L^2$ error denotes $\|u_h(t_N)-U_h^N\|_{L^2(\Omega)}$. Generally, the solution of problem
\eqref{eqn:fde} is weakly singular in time near $t=0$, even for smooth problem data.
Thus, a direct implementation of \eqref{eqn:CN} can fail to achieve the
desired rate. Even though the initial data in case \ref{exam:1d}(b) is smooth, the solution $u$ does
not have the requisite temporal regularity, giving only an $O(\tau)$ convergence,
cf. Table \ref{tab:1st}. Nonetheless, with proper corrections at initial
steps to be described below, one can restore the desired $O(\tau^2)$ rate, cf. Table \ref{tab:2nd}.

\begin{table}[htb!]
\caption{The $L^2$ error for Example \ref{exam:1d}(a) at $t_N=1$, by the scheme \eqref{eqn:CN}.}
\label{tab:smooth_sol}
\centering
     \begin{tabular}{|c|cccccc|c|}
     \hline
      $\alpha\backslash N$  &$10$ &$20$ &$40$ & $80$ & $160$ &$320$ &rate \\
            \hline
     $\al=0.1$ &6.22e-6 &1.55e-6 &3.88e-7 &9.67e-8 &2.41e-8 &5.87e-9 &$\approx$ 2.01 \\
     $\al=0.5$ &1.52e-5 &3.79e-6 &9.47e-7 &2.37e-7 &5.93e-8 &1.50e-8 &$\approx$ 2.00 \\
     $\al=0.9$   &3.84e-6 &9.63e-7 &2.42e-7 &6.08e-8 &1.54e-8 &3.95e-9 &$\approx$ 1.98 \\
      \hline
     \end{tabular}
\end{table}

\begin{table}[htb!]
\caption{The $L^2$ error for Example \ref{exam:1d}(b) at $t_N=1$, by the scheme \eqref{eqn:CN}.}
\label{tab:1st}
\centering
     \begin{tabular}{|c|cccccc|c|}
     \hline
      $\alpha\backslash N$  &$40$ &$80$ &$160$ & $320$ & $640$ &$1280$ &rate \\
            \hline
     $\al=0.1$ &1.63e-5 &8.15e-6 &4.07e-6 &2.04e-6 &1.02e-6 &5.09e-7 &$\approx$ 1.00 \\
     $\al=0.5$ &3.13e-5 &1.58e-5 &7.97e-6 &4.00e-6 &2.00e-6 &1.00e-6 &$\approx$ 0.99 \\
     $\al=0.9$   &2.04e-6 &1.35e-6 &7.55e-7 &3.98e-7 &2.05e-7&1.03e-7 &$\approx$ 0.96 \\
      \hline
     \end{tabular}
\end{table}

\begin{table}[htb!]
\caption{The $L^2$ error for Example \ref{exam:1d}(b) at $t_N=1$,
by the corrected scheme \eqref{eqn:fully-mod}.}
\label{tab:2nd}
\centering
     \begin{tabular}{|c|cccccc|c|}
     \hline
      $\alpha\backslash N$  &$10$ &$20$ &$40$ & $80$ & $160$ &$320$ &rate \\
     \hline
     $\al=0.1$ &7.70e-6 &1.80e-6   &4.37e-7 &1.08e-8 &2.66e-8 &6.60e-9 &$\approx$ 2.01 \\
     $\al=0.5$ &4.24e-5 &9.91e-6 &2.40e-6 &5.89e-7 &1.46e-7 &3.61e-8 &$\approx$ 2.02 \\
     $\al=0.9$   &4.62e-5 &9.92e-6 &2.39e-6 &5.86e-7 &1.45e-7 &3.59e-8 &$\approx$ 2.03 \\
      \hline
     \end{tabular}
\end{table}

It is well known that if $\varphi(0)\neq0$, uncorrected high-order CQs can achieve only
an $O(\tau)$ rate, which is also the case for \eqref{eqn:CN}. Inspired by \cite{CuestaLubichPalencia:2006},
we correct the scheme \eqref{eqn:CN} as follows. To derive the correction,
we define $\tilde f_h(t):=f_h(t)-f_h(0)$ and
rewrite  \eqref{eqn:fdesemidis} into
\begin{equation*}
  \begin{aligned}
    {^R\kern-.2mm\Dal} (u_h(t)-v_h) & = \Delta_h(u_h(t)-v_h) + \Delta_h v_h + \tilde{f}_h(t) + f_h(0)\\
       & = \Delta_h(u_h(t)-v_h) + \partial_t\partial_t^{-1} \Delta_hv_h + \tilde{f}_h(t) + \partial_t\partial_t^{-1}f_h(0).
  \end{aligned}
\end{equation*}
Next we apply \eqref{eqn:CN} and approximate $\partial_t\partial_t^{-1}
$by $\tilde\partial_\tau\partial_t^{-1}$, where $\tilde\partial_\tau$ denotes the second-order BDF, i.e.,
\begin{equation*}
  \begin{aligned}
  \bar\partial_\tau^\alpha (U_h-v_h)^n & = (1-\tfrac{\alpha}{2})(\Delta_h(U_h-v_h) +\tilde f_h)^n  + \tfrac{\alpha}{2}(\Delta_h(U_h-v_h) +\tilde f_h)^{n-1} \\
   & \ \ + (1-\tfrac{\alpha}{2})(\tilde\partial_\tau\partial_t^{-1} (\Delta_hv_h  +f_h(0)))^n + \tfrac{\alpha}{2}(\tilde\partial_\tau\partial_t^{-1} (\Delta_hv_h  +f_h(0)))^{n-1}.
  \end{aligned}
\end{equation*}
The purpose of keeping $\partial_t^{-1}$ intact in the discretization and using the approximation $\tilde\partial_\tau\partial_t^{-1}$
instead of $\bar\partial_\tau\partial_t^{-1}$ is to maintain the desired $O(\tau^2)$ accuracy. Letting
$1_\tau=(0,{3}/{2},1,1,\ldots)$, then $\tilde\partial_\tau\partial_t^{-1}1=1_\tau$ on the grid points
$t_n$ \cite[Section 3]{CuestaLubichPalencia:2006}, the scheme is given explicitly by
\begin{equation}\label{eqn:fully-mod}
\begin{split}
   \bPtau^\alpha (U_h-v_h) ^1 - (1-\tfrac\alpha2)\Delta_h U_h^1 -  (\tfrac12-\tfrac\alpha4) \Delta_h v_h&= 
   (1-\tfrac\alpha2) (F_h^1+\tfrac12F_h^0),\\
   \bPtau^\alpha (U_h-v_h)^2  - (1-\tfrac\alpha2)\Delta_hU_h^2 - \tfrac\alpha2 \Delta_hU_h^1  -\tfrac\alpha4\Delta_h v_h&= 
   (1-\tfrac\alpha2)F_h^2+\tfrac\alpha2F_h^1 + \tfrac\alpha4 F_h^0,\\
   \bPtau^\alpha (U_h-v_h)^n   - (1-\tfrac\alpha2)\Delta_hU_h^n -  \tfrac\alpha2 \Delta_h U_h^{n-1} &= 
    (1-\tfrac\alpha2) F_h^n + \tfrac\alpha2 F_h^{n-1} , \quad 3\le n\le N.
\end{split}
\end{equation}
It is noteworthy that the correction only changes the first two steps.


\section{Convergence analysis}\label{sec:conv}

Now we analyze the corrected scheme \eqref{eqn:fully-mod},
and discuss homogeneous and inhomogeneous problems separately, and
focus on the temporal error.

\subsection{Solution representations}

The convergence analysis relies crucially on the integral representations of the
semidiscrete Galerkin solution $w_h(t):=u_h(t)-v_h$ and fully discrete solution $W_h^n:=U_h^n-v_h$.
First, we derive a representation of the solution $w_h(t)$
by means of Laplace transform. Clearly, the function  $w_h(t)$ satisfies
\begin{equation*}
    \Dal w_h   -\Delta_hw_h = \Delta_h v_h + f_h,
\end{equation*}
with $w_h(0)=0$. Upon Laplace transform, denoted by $~~\widehat{}~~$, and using the formula
$\widehat{\partial_t^\alpha \varphi}=z^\alpha \widehat{\varphi}-z^{\alpha-1}\varphi(0)$
\cite[Lemma 2.24, pp. 98]{KilbasSrivastavaTrujillo:2006}, we obtain
\begin{equation*}
    z^\al \widehat w_h(z)  -\Delta_h\widehat w_h(z) = z^{-1} \Delta_hv_h + \widehat f_h(z).
\end{equation*}
By inverse Laplace transform, the function $w_h(t)$ can be represented by
\begin{equation}\label{eqn:semisol}
  w_h(t)=\frac{1}{2\pi \mathrm{i}}\int_{\Gamma_{\theta,\delta}}
  e^{zt}\left(- K(z) \Delta_h v_h - zK(z) \widehat f_h(z)\right) dz,
\end{equation}
with the kernel function
\begin{equation}\label{eqn:kernel}
   K(z)= -z^{-1}(z^\al    -\Delta_h)^{-1}.
\end{equation}
In the representation \eqref{eqn:semisol}, the contour $\Gamma_{\theta,\delta}$ is defined by
\begin{equation*}
  \Gamma_{\theta,\delta}=\left\{z\in \mathbb{C}: |z|=\delta, |\arg z|\le \theta\right\}\cup
  \{z\in \mathbb{C}: z=\rho e^{\pm {\rm i}\theta}, \rho\ge \delta\},
\end{equation*}
oriented with an increasing imaginary part. Throughout, we choose the angle $\theta\in(\pi/2 , \pi)$.
Since the discrete Laplacian operator $\Delta_h$ satisfies the following resolvent estimate
\cite[Chapter 6]{Thomee:2006}, \cite[Example 3.7.5 and Theorem 3.7.11]{ABHN}
\begin{equation}\label{Deltah-resolvent}
  \| (z -\Delta_h)^{-1} \| \le c z^{-1},  \quad \forall z \in \Sigma_{\theta},
\end{equation}
there exists a constant $c$ which depends only on $\theta$ and $\al$ such that
\begin{equation}\label{eqn:resol}
  \| (z^{\al}-\Delta_h)^{-1} \| \le cz^{-\al},  \quad \forall z \in \Sigma_{\theta}.
\end{equation}

Next, we derive a representation of  $W_h^n$
by means of discrete Laplace transform, i.e., generating function. Recall that for a given sequence
$(f^n)_{n=0}^\infty$, the generating function $\widetilde f(\xi)$ is defined by
$
  \widetilde f(\xi) := \sum_{n=0}^\infty f^n\xi^n.
$
Then we have the following solution representation.

\begin{proposition}\label{lem:solurep}
Let $K(z)$ be given by \eqref{eqn:kernel} and $G_h^n: =  F_h^n-F_h^0$. Then,
there exists a $\delta_0\in(0,{\pi}/{2})$ {\rm(}independent of $\tau${\rm)} such that for $\delta\in(0,\delta_0]$ and $\theta\in({\pi}/{2},{\pi}/{2}+\delta_0]$,
the fully discrete solution $W_h^n:=U_h^n-v_h$ can be represented by
\begin{equation}\label{Rep-Wh}
  W_h^n= \frac{1}{2\pi\mathrm{i} }\int_{\Gamma^\tau_{\theta,\delta}}e^{zt_{n}}
  \left(\mu(e^{-z\tau}) K( \beta_\tau(e^{-z\tau}))(-\Delta_h v_h- F_h^0) -  \beta_\tau(e^{-z\tau})K( \beta_\tau(e^{-z\tau}))\widetilde G_h(e^{-z\tau})\tau  \right)\,dz,
\end{equation}
with the contour {\rm(}oriented with an increasing imaginary part{\rm)} defined by
$
\Gamma_{\theta,\delta}^\tau :=\{ z\in \Gamma_{\theta,\delta}:|\Im(z)|\le {\pi}/{\tau} \} .
$
The functions $\beta_\tau(\xi)$ and $\mu(\xi) $ are, respectively, given by
\begin{equation}\label{eqn:fun0}
   \beta_\tau(\xi) = \frac{1-\xi}{\tau(1-\tfrac\al2+\tfrac\al2\xi)^{1/\al}}\qquad\text{and}\qquad\mu(\xi)= \frac{3\xi-\xi^2}{2(1-\tfrac\al2+\tfrac\al2\xi)^{1/\al}}.
\end{equation}
\end{proposition}
\begin{proof}
It follows from the scheme \eqref{eqn:fully-mod} that the function $W_h^n$ satisfies
\begin{equation}\label{Eq_Wnh}
\begin{split}
  \bPtau^\alpha W_h^1   -(1-\tfrac\alpha2)\Delta_hW_h^1 -  (\tfrac32-\tfrac34\alpha) \Delta_h v_h &= (1-\tfrac\alpha2) (F_h^1+\tfrac12F_h^0),\\
   \bPtau^\alpha W_h^2 - (1-\tfrac\alpha2)\Delta_hW_h^2  -\tfrac\alpha2\Delta_hW_h^1-(1+\tfrac\alpha4) \Delta_h v_h&= (1-\tfrac\alpha2)F_h^2+\tfrac\alpha2F_h^1 + \tfrac\alpha4 F_h^0,\\
  \bPtau^\alpha W_h^n  - (1-\tfrac\alpha2)\Delta_hW_h^n    -\tfrac\alpha2 \Delta_h W_h^{n-1}  -\Delta_h v_h&=  (1-\tfrac\alpha2) F_h^n + \tfrac\alpha2 F_h^{n-1} , \quad 3\le n\le N.
\end{split}
\end{equation}
with $W_h^0=0$. By multiplying both sides by $\xi^n$ and summing up the results for $n=1,2,\dots$, we obtain
\begin{equation}\label{Eq_Wnh-0}
  \begin{aligned}
     &\sum_{n=1}^\infty \xi^n \bPtau^\alpha W_h^n - \sum_{n=1}^\infty( (1-\tfrac\alpha2)\Delta_hW_h^n +\tfrac\alpha2\Delta_h W_n^{n-1})\xi^n-\Delta_hv_h\bigg(\sum_{n=1}^\infty \xi^n +(\tfrac12-\tfrac{3\alpha}{4})\xi+\tfrac\alpha4\xi^2\bigg) \\
   &= \sum_{n=1}^\infty\left((1-\tfrac\alpha2)F_h^n+\tfrac\alpha2F_h^{n-1} \right)\xi^n + \left((\tfrac{1}{2}-\tfrac{3\alpha}{4})\xi  + \tfrac{\alpha}{4}\xi^2\right)F_h^0.
  \end{aligned}
\end{equation}
Next we simplify the summations. Since $W_h^0=0$, by the discrete convolution rule, we deduce
\begin{equation*}
  \begin{aligned}
  \sum_{n=1}^\infty \xi^n\bPtau^\alpha W_h^n &= \sum_{n=0}^\infty \xi^n\bPtau^\alpha W_h^n = \tau^{-\alpha}(1-\xi)^\alpha\widetilde{W}_h(\xi),\\
  \sum_{n=1}^\infty( (1-\tfrac\alpha2)\Delta_hW_h^n &+\tfrac\alpha2\Delta_hW_n^{n-1})\xi^n = ((1-\tfrac\alpha2)+\tfrac{\alpha}{2}\xi)\Delta_h\widetilde W_h(\xi).
  \end{aligned}
\end{equation*}
Meanwhile, by a simple computation, we have
$
   \sum_{n=1}^\infty \xi^n +(\frac12-\frac{3\alpha}{4})\xi+\frac\alpha4\xi^2 
    = \frac{3\xi-\xi^2}{2(1-\xi)}(1-\frac{\alpha}{2}+\frac{\alpha}{2}\xi).
$
Consequently, the definition of $G_h^n$ implies $G_h^0=0$ and
\begin{equation*}
  \begin{aligned}
   &\sum_{n=1}^\infty\left((1-\tfrac\alpha2)F_h^n+\tfrac\alpha2F_h^{n-1} \right)\xi^n + \left((\tfrac{1}{2}-\tfrac{3\alpha}{4})\xi  + \tfrac{\alpha}{4}\xi^2\right)F_h^0\\
   = & \sum_{n=1}^\infty\left((1-\tfrac\alpha2)G_h^n+\tfrac\alpha2G_h^{n-1} \right)\xi^n + \left(\sum_{n=1}^\infty \xi^n+(\tfrac{1}{2}-\tfrac{3\alpha}{4})\xi  + \tfrac{\alpha}{4}\xi^2\right)F_h^0\\
   = & (1-\tfrac{\alpha}{2}+\tfrac{\alpha}{2}\xi)\widetilde G_h(\xi) + \tfrac{3\xi-\xi^2}{2(1-\xi)}(1-\tfrac{\alpha}{2}+\tfrac{\alpha}{2}\xi)F_h^0.
  \end{aligned}
\end{equation*}
Substituting the preceding identities into \eqref{Eq_Wnh-0} yields
\begin{equation*}
    \left((\beta_\tau(\xi))^{\al}   -\Delta_h  \right) \widetilde W_h(\xi) = \kappa(\xi)\Delta_h v_h
    + \kappa(\xi)  F_h^0 + \widetilde G_h(\xi),
\end{equation*}
with $\kappa(\xi) = \tfrac{3\xi-\xi^2}{2(1-\xi)}$. Since $|\xi|\leq 1$,
$\beta_\tau(\xi)^\alpha \in \Sigma_{\theta^\prime}$ for some $\theta^\prime\in(\pi/2,\pi)$ \cite[proof of Theorem 6.1]{JLZ},
by the resolvent estimate \eqref{eqn:resol}, we have
\begin{equation}\label{Def-tilde-Wh}
\widetilde W_h(\xi) =  \left((\beta_\tau(\xi))^{\al} -\Delta_h  \right)^{-1}\left(\kappa(\xi) \Delta_h v_h + \kappa(\xi) F_h^0 + \widetilde G_h (\xi)\right).
\end{equation}
Without loss of generality, we can assume $F^n_h=F^0_h$ (so $G^n_h=0$) for $n> N=T/\tau$.
Otherwise we redefine $F^n_h:=F^0_h$ for $n> N=T/\tau$, and this modification does not affect
of the value of $W_h^n$ for $n=1,\dots,N$, in view of \eqref{Eq_Wnh}. Clearly, the function
$\widetilde  W_h(\xi)$ defined in \eqref{Def-tilde-Wh} is analytic with respect to $\xi$ in a neighborhood of
the origin, and thus Cauchy's integral formula implies that for small $\varrho$
\begin{equation*}
    W_h^n = \frac{1}{2 \pi\mathrm{i}}\int_{|\xi|=\varrho}\xi^{-n-1} \widetilde  W_h(\xi)  d\xi
    = \frac{\tau}{2\pi\mathrm{i}}\int_{\Gamma^\tau} e^{zt_{n}}\widetilde  W_h(e^{-z\tau})\, dz,
\end{equation*}
where the second equality follows by changing the variables $\xi=e^{-z\tau}$, and the contour $\Gamma^\tau$ is given by
\begin{equation*}
   \Gamma^\tau:=\{ z=-\ln(\varrho)/\tau+\mathrm{i} y: \, y\in{\mathbb R}\,\,\,\mbox{and}\,\,\,|y|\le {\pi}/{\tau} \}.
\end{equation*}
Since both $\kappa(e^{-z\tau})$ and $\widetilde G_h(e^{-z\tau})$ are analytic with respect to $z\in{\mathbb C}
\backslash\{0\}$, Lemma \ref{lem:est-hom-2} below implies that the function $e^{zt_{n}}
\widetilde  W_h(e^{-z\tau})$  is analytic with respect
to $z$ in the region enclosed by $\Gamma^\tau$, $\Gamma^\tau_{\theta,\delta}$ and the two lines $\Gamma_{\pm}^\tau:
={\mathbb R}\pm \mathrm{i}\pi/\tau$ (oriented from left to right). Then, since the values of $e^{zt_{n}}\widetilde  W_h
(e^{-z\tau})$ on the two lines $\Gamma_{\pm}^\tau$ coincide, it follows from Cauchy's theorem that
\begin{align*}
\begin{aligned}
    W_h^n
    & = \frac{\tau}{2\pi\mathrm{i}}\int_{\Gamma^\tau} e^{zt_{n}}\widetilde  W_h(e^{-z\tau})\, dz  = \frac{\tau}{2\pi\mathrm{i}}\int_{\Gamma^\tau_{\theta,\delta}} e^{zt_{n}}\widetilde  W_h(e^{-z\tau})\, dz\\
    &\quad     +\frac{\tau}{2\pi\mathrm{i}}\int_{\Gamma^\tau_+} e^{zt_{n}}\widetilde  W_h(e^{-z\tau})\, dz -\frac{\tau}{2\pi\mathrm{i}}\int_{\Gamma^\tau_-} e^{zt_{n}}\widetilde  W_h(e^{-z\tau})\, dz
    = \frac{\tau}{2\pi\mathrm{i}}\int_{\Gamma^\tau_{\theta,\delta}} e^{zt_{n}}\widetilde  W_h(e^{-z\tau})\, dz .
\end{aligned}
\end{align*}
This completes the proof of the proposition.
\end{proof}

The next result gives basic estimates on the functions
$(1-\frac{\alpha}{2}+\frac{\alpha}{2}e^{-z\tau})^{1/\alpha}$ and $(1-e^{-z\tau})^\alpha$
from Proposition \ref{lem:solurep}. These estimates are crucial for the error analysis in Section \ref{ssec:homo} below.
\begin{lemma}\label{lem:g}
Let $\alpha\in(0,1)$. Then there exists a $\delta_1>0$ {\rm(}independent of $\tau${\rm)}
such that for $\delta\in(0,\delta_1]$ and $\theta\in({\pi}/{2},{\pi}/{2}+\delta_1]$, there hold for
any $z\in \Gamma_{\theta,\delta}^\tau $
  \begin{align}
     c_0\leq  |(1-\tfrac{\alpha}{2}+\tfrac{\alpha}{2}e^{-z\tau})^{1/\alpha}|&\leq c_1,\label{eqn:est-g-1}\\
    |(1-\tfrac{\alpha}{2}+\tfrac{\alpha}{2}e^{-z\tau})^{1/\alpha}-(1-\tfrac{z\tau}{2})|&\leq c\tau^2|z|^2,\label{eqn:est-g-2} \\
    |(1-e^{-z\tau})^\alpha - \tau^\alpha z^\alpha(1-\tfrac{\alpha}{2}+\tfrac{\alpha}{2}e^{-z\tau})|&\leq c|z|^{2+\alpha}\tau^{2+\alpha},\label{eqn:est-g-3}
  \end{align}
where the constants $c_0$, $c_1$ and $c$ are independent of $\tau$, $\theta$ and $\delta$ {\rm(}but may depend on $\delta_1${\rm)}.
\end{lemma}
\begin{proof}
Let $g(z)=(1-\frac{\alpha}{2}+\frac{\alpha}{2}e^{-z\tau})^{1/\alpha}$.
First we consider $z\in \Gamma_{\theta,+}^\tau$, and  write $z=re^{\mathrm{i}\theta}$, $r\in(\delta,\pi/(\tau\sin\theta)]$. Then
with $s=r\tau\sin\theta\in (0,\pi)$ and $\gamma=-\cot \theta>0$, $\eta=\frac{2}{\alpha}-1>1$, there holds
\begin{equation*}
   |g(z)|^\alpha = \tfrac{\alpha}{2}(\eta^2 + e^{2\gamma s} + 2\eta e^{\gamma s}\cos s)^{1/2}
     \geq \tfrac{\alpha}{2}(\eta^2 + e^{2\gamma s}-2 e^{\gamma s})^{1/2}
     \geq \tfrac{\alpha}{2}(\eta -e^{\gamma\pi}).
\end{equation*}
Since $\alpha\in(0,1)$, we have $\eta-e^{\gamma\pi}>0$, for $\theta\in(\pi/2,\pi)$ close to $\pi/2$.
Next we consider $z\in \Gamma_\delta$, with $z=\delta e^{\mathrm{i}\varphi}$,
$0\leq \varphi\leq \theta$ and small $\delta$. Then by letting $\rho=\tau\delta \in (0,1)$ and $s=\rho\cos\varphi \in[-\epsilon,\rho]$, for small $\epsilon>0$,
and $h(s)=(\rho^2-s^2)^{1/2}\leq \rho\leq 1$, we have $\cos h(s)\geq 0$ and thus
\begin{equation*}
     |g(z)|^\alpha
      =\tfrac{\alpha}{2}(\eta^2+2\eta e^{-s}\cos h + e^{-2s})^{1/2}\geq \tfrac{\alpha}{2}(\eta^2 + e^{-2s})^{1/2} \geq \tfrac{\alpha}{2}\eta.
\end{equation*}
This shows the lower bound on $|g(z)|$ in \eqref{eqn:est-g-1}. The upper bound on $|g(z)|$ in \eqref{eqn:est-g-1} follows by
$|g(z)|^\alpha \le 1-\tfrac\al2 + \tfrac\al2 e^{-\pi \cot\theta} \le c$ for any $z\in \Gamma_{\theta,+}^\tau,$
and a similar bound for $z\in \Gamma_\delta$.

For the  estimate \eqref{eqn:est-g-2}, it suffices to show
\begin{equation}\label{eqn:est-g}
  |g(z)-(1-\tfrac{z\tau}{2})|/(|z|^2\tau^2) \leq c, \quad \forall z\in \Gamma_{\theta,\delta}^\tau .
\end{equation}
If $|z|\tau\leq\epsilon$, where $\epsilon\in(0,1)$ is to be determined, then by Taylor expansion, we deduce
\begin{equation*}
  g(z) - (1-\tfrac{z\tau}{2}) = \sum_{k=2}^\infty C(\tfrac{1}{\alpha},k)(-\tfrac{\alpha}{2}+\tfrac{\alpha}{2}e^{-z\tau})^k + O(|z|^2\tau^2),
\end{equation*}
with $C(\gamma,k)=\frac{\Gamma(\gamma+1)}{\Gamma(k+1)\Gamma(\gamma-k+1)}
$. Meanwhile, we have
\begin{equation*}
  \begin{aligned}
   |-\frac{\alpha}{2} + \frac{\alpha}{2}e^{-z\tau}| &\leq \frac{\alpha}{2}|z\tau|\sum_{k=0}^\infty \frac{|z\tau|^{k}}{(k+1)!}
      \leq \frac{\alpha}{2} |z|\tau \frac{e^{\epsilon}-1}{\epsilon}.
  \end{aligned}
\end{equation*}
Since $f(\epsilon)=\frac{e^\epsilon-1}{\epsilon}$ is increasing in $\epsilon$ for $\epsilon\in(0,\infty)$ and
$\lim_{\epsilon\to0^+} f(\epsilon)=1$, for any $\alpha\in(0,1)$, there exists an $\epsilon\in(0,1)$
such that $\frac{\alpha(e^\epsilon-1)}{2\epsilon} <1$. By ratio test, $
|\sum_{k=2}^\infty C(\tfrac{1}{\alpha},k)(-\tfrac{\alpha}{2}+\tfrac{\alpha}{2}e^{-z\tau})^k|\leq c|z|^2\tau^2$, and
thus \eqref{eqn:est-g} holds. Meanwhile, for $|z|\tau>\epsilon$, there exists a $\delta_1>0$ (independent of $\tau$)
such that for $\delta\in(0,\delta_1]$ and $\theta\in({\pi}/{2},{\pi}/{2}+\delta_1]$,
$|g(z)|\leq c$. Since $| z|\tau\leq \pi/\sin\theta$ for
$z\in \Gamma_{\theta,\delta}^\tau$, this again yields \eqref{eqn:est-g}, showing the estimate \eqref{eqn:est-g-2}.

Next we turn to the third estimate \eqref{eqn:est-g-3}. Since $|z|\tau\leq c$ for
$z\in\Gamma_{\theta,\delta}^\tau$, like before, it suffices to show \eqref{eqn:est-g-3}
for $|z|\tau\leq 1$. For $|z|\tau\leq 1$, by Taylor expansion, we deduce
\begin{equation*}
  1-e^{-z\tau} = z\tau \sum_{j=1}^\infty\frac{(-z\tau)^{j-1}}{j!} = z\tau + z\tau \sum_{j=2}^\infty\frac{(-z\tau)^{j-1}}{j!}.
\end{equation*}
In the identity
$
  \sum_{j=2}^\infty \frac{(-z\tau)^{j-1}}{j!} = \frac{-z\tau}{2} + (-z\tau)^2\sum_{j=3}^\infty\frac{(-z\tau)^{j-2}}{j!},
$
we have
\begin{equation*}
  |\sum_{j=3}^\infty \frac{(-z\tau)^{j-2}}{j!}| \leq \sum_{j=3}^\infty\frac{1}{j!}\leq e\quad \mbox{and}\quad
  |\sum_{j=2}^\infty \frac{(-z\tau)^{j-1}}{j!}| \leq 
  |z|\tau(e-2)<|z|\tau.
\end{equation*}
Thus we have
\begin{equation*}
  \begin{aligned}
    (1-e^{-z\tau})^\alpha &= z^\alpha\tau^\alpha \big(1+ \sum_{j=2}^\infty\frac{(-z\tau)^{j-1}}{j!}\big)^\alpha \\
      & = z^\alpha \tau^\alpha  + \alpha z^\alpha \tau^\alpha \sum_{j=2}^\infty \frac{(-z\tau)^{j-1}}{j!} + z^\alpha\tau^\alpha\sum_{k=2}^\infty C(\alpha,k)\big(\sum_{j=2}^\infty\frac{(-z\tau)^{j-1}}{j!}\big)^k\\
      & = z^\alpha \tau^\alpha - \tfrac{\alpha}{2}z^{\alpha+1}\tau^{\alpha+1} + O(|z|^{\alpha+2}\tau^{\alpha+2}),
  \end{aligned}
\end{equation*}
and $  \tau^\alpha z^\alpha (1-\tfrac{\alpha}{2}+\tfrac{\alpha}{2}e^{-z\tau}) = \tau^\alpha z^\alpha -\tfrac{\alpha}{2}\tau^{\alpha+1}z^{\alpha+1} +O(|z|^{\alpha+2}\tau^{\alpha+2}).$
Combining the last two estimates completes the proof of the lemma.
\end{proof}

The next lemma gives a crucial sector mapping property of the function $\beta_\tau(e^{-z\tau})^\alpha$.
The proof relies on the fact that $\beta_\tau(e^{-z\tau})^\alpha$ is
very close to $\beta_\tau(e^{-\mathrm{i}s})^\alpha$ for $z\in\Gamma_{\theta,+}^\tau$
(if $\theta\in(\pi/2,\pi)$ is sufficiently close to $\pi/2$) and uses the result $\beta_\tau(e^{-\mathrm{i}s})^\alpha
\in \Sigma_{\alpha\pi/2}$ from \cite[Theorem 6.1]{JLZ}.

\begin{lemma}\label{lem:est-hom-2}
For $\alpha\in(0,1)$, let $\phi\in(\alpha\pi/2,\pi)$ be fixed.
Then there exists a $\delta_{0}>0$ {\rm(}independent of $\tau${\rm)} such that for $\delta\in(0,\delta_{0}]$ and
$\theta\in({\pi}/{2},{\pi}/{2}+\delta_{0}]$, we have
\begin{align}\label{angle_beta-tau}
&\beta_\tau(e^{-z\tau})^\alpha\in \Sigma_{\phi} \, \quad \forall \, z\in \Gamma^\tau_{\theta,\delta}  \, \cup \, \overline\Sigma_{\pi/2}\backslash\{0\} .
\end{align}
Moreover, the operator $(\beta_\tau(e^{-z\tau})^\al-\Delta_h)^{-1}$ is analytic with respect to $z$ in the region enclosed by the curves $\Gamma^\tau$, $\Gamma^\tau_{\theta,\delta}$ and $\Gamma_{\pm}^\tau:={\mathbb R}\pm \mathrm{i}\pi/\tau$, and satisfies
\begin{equation} \label{beta:resolv}
    \|(\beta_\tau(e^{-z\tau})^\al-\Delta_h)^{-1}\|\le c |\beta_\tau(e^{-z\tau})|^{-\alpha} ,\quad \forall z\in \Gamma_{\theta,\delta}^\tau \, ,
\end{equation}
where the constant $c$ is independent of $\tau$ {\rm(}but may depend on $\phi${\rm)}.
\end{lemma}
\begin{proof}
For the proof, we split the contour $\Gamma_{\theta,\delta}^\tau$ into two parts, i.e.,
\begin{equation}\label{Def-Gamma_theta_delta-tau}
  \Gamma_{\theta,\delta}^\tau:=\Gamma_{\delta}\cup\Gamma_{\theta,\pm}^\tau:=\left\{z\in \mathbb{C}: |z|=\delta, |\arg z|\le \theta\right\}\cup
  \big\{z\in \mathbb{C}: z=r e^{\pm \mathrm{i}\theta},\,\,  \delta\le r \le {\pi}/({\tau|\sin(\theta)|})\big\} .
\end{equation}
To prove \eqref{angle_beta-tau}, we consider the following three cases $z\in \Gamma_\delta$, $z\in \Gamma_{\theta,\pm}^\tau$ and $z\in \overline\Sigma_{\pi/2}\backslash\{0\}$, separately. First, for $z\in\Gamma_\delta
\subset\Sigma_{\theta}$, by choosing $\delta>0$ sufficiently small and using Taylor's  expansion, we have
\begin{equation*}
  \beta_\tau(e^{-z\tau})^\alpha
  = \frac{(1-e^{-z\tau})^\alpha}{\tau^\alpha(1-\tfrac\al2+\tfrac\al2e^{-z\tau})}
  =|z|^\alpha  e^{{\rm i}\alpha\,{\rm arg}(z)} (1+O(z\tau))
  \in \Sigma_{\alpha\theta+\varepsilon_\delta},
\end{equation*}
for some $\varepsilon_\delta>0$ with
$\lim_{\delta\rightarrow 0^+}\varepsilon_\delta=0$, showing the relation \eqref{angle_beta-tau}
for $z\in\Gamma_\delta$. Second, for $z=|z|e^{\mathrm{i}\theta}\in\Gamma_{\theta,+}^\tau$, we
have $e^{-z\tau}=e^{-s \cot(\theta)}e^{-\mathrm{i} s}$,
$s=|z|\tau \sin(\theta)\in(0,\pi)$. Let $\gamma_\tau(\xi):=\beta_\tau(\xi)^\alpha$. Then
there exists some $\sigma(s)\in(0,1)$ such that
\begin{equation*}
  |\beta_\tau(e^{-z\tau})^\alpha-\beta_\tau(e^{-\mathrm{i} s})^\alpha|
 =|\gamma_\tau(e^{-z\tau}) -\gamma_\tau(e^{-\mathrm{i} s})|
  \le cs|\cot\theta||\gamma_\tau'(e^{-\sigma(s) s\cot(\theta)}e^{-\mathrm{i} s}) |
\end{equation*}
Straightforward computation gives $\gamma_\tau^\prime(\xi)=-\alpha\tau^{-\alpha}\frac{(1-\xi)^{\alpha-1}(3-\alpha
+(\alpha-1)\xi)}{2(1-\frac{\alpha}{2}+\frac{\alpha}{2}\xi)^2}$. For $\theta\in(\pi/2,\pi)$
sufficiently close to $\pi/2$, $e^{-\sigma(s) s\cot(\theta)}\approx 1$. Then Lemma \ref{lem:g} implies
\begin{equation*}
  |\gamma_\tau'(e^{-\sigma(s) s\cot(\theta)}e^{-\mathrm{i} s})|
  \leq c\tau^{-\alpha}|1-e^{-\sigma(s)s\cot\theta}e^{-\mathrm{i}s}|^{\alpha-1}.
\end{equation*}
Combining the preceding two estimates with the inequality $|\cot\theta| \leq c|\theta-\pi/2|$ yields
\begin{equation*}
  |\beta_\tau(e^{-z\tau})^\alpha-\beta_\tau(e^{-\mathrm{i} s})^\alpha| \leq c\tau^{-\alpha}|\theta-\pi/2|s|1-e^{-\sigma(s)s\cot\theta}e^{-\mathrm{i}s}|^{\alpha-1}.
\end{equation*}
If $s\in(0,\pi)$ is small, then Taylor's expansion yields $\beta_\tau
(e^{-\mathrm{i}s})^\alpha\approx \tau^{-\alpha}s^\alpha e^{\mathrm{i}\alpha\pi/2}$ and
$1-e^{-s\sigma(s)\cot(\theta)}e^{-\mathrm{i}s}\approx s\sigma(s)\cot(\theta)+\mathrm{i}s$
asymptotically. Consequently, we have
\begin{equation*}
  |\beta_\tau(e^{-z\tau})^\alpha-\beta_\tau(e^{-\mathrm{i} s})^\alpha|
  \le c\tau^{-\alpha}|\theta-\pi/2|s^{\alpha} \le c|\theta-\pi/2||\beta_\tau(e^{-\mathrm{i}s})^\alpha| .
\end{equation*}
Since $\beta_\tau(e^{-\mathrm{i}s})^\alpha\in\Sigma_{\alpha\pi/2}$ \cite[Proof of Theorem 6.1]{JLZ},
it follows that $\beta_\tau(e^{-z\tau})^\alpha\in \Sigma_{\phi}$ when $s$ is sufficiently small.
Meanwhile, if $s\in(0,\pi)$ is away from $0$, then $|\beta_\tau(e^{-\mathrm{i}s})^\alpha|\ge c\tau^{-\alpha}$ and so
\begin{equation*}
  |\beta_\tau(e^{-z\tau})^\alpha-\beta_\tau(e^{-\mathrm{i} s})^\alpha|
  \le c|\theta-\pi/2|\tau^{-\alpha}\le c|\theta-\pi/2| |\beta_\tau(e^{-\mathrm{i}s})^\alpha| .
\end{equation*}
By choosing $\theta\in(\pi/2,\pi)$ sufficiently close to $\pi/2$, we again have $\beta_\tau(e^{-z\tau})^\alpha\in \Sigma_{\phi}$.
The proof for the case $z=|z|e^{\mathrm{i}\theta}\in\Gamma_{\theta,-}^\tau$ is similar as the case $\Gamma_{\theta,+}^\tau$ and thus omitted.
Third and last, for $z\in \overline\Sigma_{\pi/2}\backslash\{0\}$, we have
$|e^{-z\tau}|\le 1$. In this case, \cite[Proof of Theorem 6.1]{JLZ} implies
$$
\beta_\tau(e^{-z\tau})^\alpha\in \Sigma_{\alpha\pi/2}\subset\Sigma_{\phi}.
$$
Next we show the analyticity. Since the spectrum of the operator $\Delta_h$ is contained in the
negative part of the real line, the result \eqref{angle_beta-tau} (with arbitrary
$\delta\in(0,\delta_0]$ and $\theta\in({\pi}/{2},{\pi}/{2}+\delta_0]$) implies
that the operator $(\beta_\tau(e^{-z\tau})^\alpha-\Delta_h)^{-1}$ is analytic with
respect to $z$ on the right side of the curve $\Gamma_{\theta_0,\delta_0}^\tau$,
with $\theta_0:= {\pi}/{2}+\delta_0$. The resolvent estimate \eqref{beta:resolv}
follows immediately from \eqref{Deltah-resolvent} and \eqref{angle_beta-tau}.
\end{proof}

\subsection{Error analysis for the homogeneous problem}\label{ssec:homo}
First we analyze the homogeneous problem, i.e., $f\equiv0$. By \eqref{eqn:semisol} and
Proposition \ref{lem:solurep}, we have
\begin{equation*}
  \begin{aligned}
     w_h(t_n) &=-\frac{1}{2\pi \mathrm{i}}\int_{\Gamma_{\theta,\delta}} e^{zt_n}  K(z)\Delta_h v_h  dz \quad\mbox{and}\quad
     W_h^n &= -\frac{1}{2\pi \mathrm{i}}\int_{\Gamma^{\tau}_{\theta,\delta}} e^{zt_{n}}
 \mu(e^{-z\tau}) K( \beta_\tau(e^{-z\tau}))\Delta_h v_h dz.
  \end{aligned}
\end{equation*}
Hence, the convergence analysis hinges on properly bounding the approximation error of the
kernel $K( \beta_\tau(e^{-z\tau}))$ to $K(z)$ along the contour $\Gamma_{\theta,\delta}^\tau$.
The next lemma provides the crucial estimate on $\mu$ and $\beta_\tau$.
\begin{lemma}\label{lem:est-hom-1}
Let $\alpha\in(0,1)$ be given, and $\mu(\xi)$, $\beta_\tau(\xi)$ be defined as \eqref{eqn:fun0} and
the constant $\delta_1$ be given in Lemma \ref{lem:g}. Then for $\delta\in(0,\delta_1]$ and
$\theta\in({\pi}/{2},{\pi}/{2}+\delta_1]$, we have for any $z\in \Gamma_{\theta,\delta}^\tau $
\begin{equation}\label{eqn:beta-mu}
  |\mu(e^{-z\tau})-1| \le c \tau^2 |z|^2,\quad
  |\beta_\tau(e^{-z\tau})  - z | \le c \tau^{2} |z|^{3},\quad
 \mbox{and}\quad |\beta_\tau(e^{-z\tau})^\al - z^\al| \le c \tau^{2} |z|^{2+\al},
\end{equation}
and
\begin{equation}
    c_0 |z| \le |\beta_\tau(e^{-z\tau})| \le c_1 |z|. \label{abs_beta}
\end{equation}
The constants $c_0$, $c_1$ and $c$
are independent of $\tau$, $\theta$ and $\delta$ {\rm(}but may depend on $\delta_1${\rm)}.
\end{lemma}
\begin{proof}
The three estimates in \eqref{eqn:beta-mu} are direct consequences of Lemma \ref{lem:g}.
The upper bound in \eqref{abs_beta} follows from
Lemma \ref{lem:est-hom-1} and the triangle equality
\begin{align*}
  \begin{aligned}
    |\beta_\tau(e^{-z\tau})|
    &\le (|\beta_\tau(e^{-z\tau}) -z | + |z| ) \le (1+ c^2\tau^2 |z|^2) |z| \\
    &\le
    \left\{\begin{array}{ll}
    (1+c^2\tau^2\delta^2)  |z|, &\mbox{for}\,\,\,z\in\Gamma_\delta,\\
    (1+c^2(\pi/\sin\theta)^2)  |z|, &\mbox{for}\,\,\,z\in\Gamma_\theta .
    \end{array}\right.
  \end{aligned}
\end{align*}
Since $c_0 |z| \le |\tfrac{1-e^{-z\tau}}{\tau}|\le c_1|z|$ \cite[Lemma 3.1]{JinLazarovZhou:2016ima},
the lower bound in \eqref{abs_beta} follows from the fact that $| 1-\tfrac\al2+\tfrac\al2 e^{-z\tau } |$ is
uniformly bounded from below in $\tau$ for all $z\in \Gamma_{\theta,\delta}^\tau $, cf. Lemma \ref{lem:g}.
\end{proof}

Using Lemmas \ref{lem:est-hom-2} and \ref{lem:est-hom-1}, we have the following error estimate of the
kernel $K(\beta_\tau(e^{-z\tau}))$.
\begin{lemma}\label{lem:est-hom-3}
Let $\delta_0$ and $\delta_1$ be defined in Proposition \ref{lem:solurep} and Lemma \ref{lem:g}, respectively.
Then by choosing $\delta=\min(\delta_0,\delta_1)$ and $\theta={\pi}/{2}+\delta$, we have
\begin{equation*}
\| \mu(e^{-z\tau}) K( \beta_\tau(e^{-z\tau}))- K(z) \| \le   c \tau^2 |z|^{1-\al},\quad \forall\,z\in \Gamma_{\theta,\delta}^\tau ,
\end{equation*}
where the constant $c$ is independent of $\tau$.
\end{lemma}
\begin{proof}
By the triangle inequality, we obtain
\begin{equation*}
\begin{split}
     \| \mu(e^{-z\tau}) K( \beta_\tau(e^{-z\tau}))- K(z) \|
     \le  | \mu(e^{-z\tau})-1| \|  K(z) \| + |\mu(e^{-z\tau})| \| K( \beta_\tau(e^{-z\tau}))- K(z) \|=:\mathrm{ I}+ \mathrm{II}.
\end{split}
\end{equation*}
The bound on the first term $\mathrm{I}$ follows from
\eqref{eqn:resol} and Lemma \ref{lem:est-hom-1}. Appealing to
Lemma \ref{lem:est-hom-1} again yields
\begin{equation}\label{eqn:difc1}
 |\beta_\tau(e^{-z\tau})^{-1}-z^{-1}|= |z-\beta_\tau(e^{-z\tau})| |\beta_\tau(e^{-z\tau})|^{-1} |z|^{-1} \le c\tau^2|z|.
\end{equation}
Similarly, by using \eqref{eqn:resol} and \eqref{beta:resolv}, and Lemma \ref{lem:est-hom-1}, and the
identity $(\beta_\tau(e^{-z\tau})^\al-\Delta_h)^{-1}-(z^\al-\Delta_h)^{-1} = (z^\alpha-\beta_{\tau}(e^{-z\tau}))(\beta_\tau(e^{-z\tau})^\alpha-\Delta_h)^{-1}(z^\alpha-\Delta_h)^{-1} $, we obtain
\begin{equation}\label{eqn:difc2}
\begin{aligned}
 & \| (\beta_\tau(e^{-z\tau})^\al-\Delta_h)^{-1}-(z^\al-\Delta_h)^{-1} \| \\
 \le& |\beta_\tau(e^{-z\tau})^\al - z^\al |
 \|(\beta_\tau(e^{-z\tau})^\al-\Delta_h)^{-1}\|\|(z^\al-\Delta_h)^{-1} \| \\
 \le& c\tau^2|z|^{2+\alpha}\|(\beta_\tau(e^{-z\tau})^\al-\Delta_h)^{-1}\|
 \|(z^\al-\Delta_h)^{-1} \| \le c\tau^2 |z|^{2-\al},
\end{aligned}
\end{equation}
and hence, the second term $\mathrm{II}$ can be bounded by
\begin{equation*}
\mathrm{II} \le c|\beta_\tau(e^{-z\tau})^{-1}-z^{-1}| \| (z^\al-\Delta_h)^{-1}  \| + c |z|^{-1}\| (\beta_\tau(e^{-z\tau})^\al-\Delta_h)^{-1}-(z^\al-\Delta_h)^{-1}\|
 \le c\tau^2 |z|^{1-\al},
\end{equation*}
which completes the proof of the lemma.
\end{proof}

Now we state the temporal error for smooth initial data $v\in D(\Delta)$.
\begin{theorem}\label{thm:v-smooth}
Let $f=0$, and $u_h$ and $U_h^n$ be the solutions of \eqref{eqn:fdesemidis} and \eqref{eqn:fully-mod},
respectively, with $v\in D(\Delta)$ and $U_h^0= v_h:=R_hv$. Then there holds
\begin{equation*}
   \| u_h(t_n)-U_h^n \|_{L^2(\Omega)} \le c  \tau^2 t_n^{\al-2} \| \Delta v \|_{L^2(\Omega)}, \quad n \ge 1.
\end{equation*}
\end{theorem}
\begin{proof}
With the constants $\delta_0$ and $\delta_1$ given in Proposition \ref{lem:solurep} and Lemma
\ref{lem:g}, respectively, we choose $\delta=\min(\delta_0,\delta_1)$ and $\theta=
{\pi}/{2}+\delta$. By \eqref{eqn:semisol} and Proposition \ref{lem:solurep}, we split
the error into
\begin{equation*}
\begin{split}
 u_h(t_n)-U_h^n
 &=  -\frac{1}{2\pi\mathrm{i}} \int_{\Gamma_{\theta,\delta}^\tau } e^{zt_n}
  \left( K(z)-\mu(e^{-zt})K(\beta_\tau(e^{-z\tau}))\right)\Delta_h v_h\,dz\\
  &\quad -\frac{1}{2\pi\mathrm{i}} \int_{\Gamma_{\theta,\delta}\backslash\Gamma_{\theta,\delta}^\tau } e^{zt_n}K(z)\Delta_h v_h dz
  =: \mathrm{I} + \mathrm{II}.
\end{split}
\end{equation*}
By Lemma \ref{lem:est-hom-3} and choosing $\delta \le 1/t_n$, we bound the first term $\mathrm{I}$ by
\begin{equation*}
\begin{split}
    \| \mathrm{I} \|_{L^2(\Omega)} &\le c \tau^2 \| \Delta_hv_h \|_{L^2(\Omega)}\bigg(\int_{\delta}^{\pi/(\tau\sin\theta)}  e^{r t_n\cos\theta} r^{1-\alpha}dr
     + \int_{-\theta}^{\theta}   e^{\delta t_n|\cos\psi|} \delta^{2-\alpha}d\psi \bigg)\\
     &\le c (t_n^{\alpha-2}+\delta^{2-\alpha})\tau^2 \|\Delta_hv_h\|_{L^2(\Omega)}
    \le c \tau^2 t_n^{\alpha-2} \|\Delta_hv_h\|_{L^2(\Omega)} .
\end{split}
\end{equation*}
For the second term $\mathrm{II}$, by the estimate \eqref{eqn:resol} and the
change of variables $s=rt_n$, we obtain
\begin{equation*}
\begin{split}
    \| \mathrm{II} \|_{L^2(\Omega)} &\le c \| \Delta_hv_h \|_{L^2(\Omega)} \int_{\pi/(\tau\sin\theta)}^\infty  e^{r t_n\cos\theta} r^{-\alpha-1} \,dr\\
    &{\le c \tau^2 \| \Delta_hv_h \|_{L^2(\Omega)} \int_{0}^\infty  e^{r t_n\cos\theta} r^{1-\alpha} \,dr\qquad (\because r\geq \pi/(\tau\sin\theta))}\\
    &{ \leq c\tau^2 t_n^{\alpha-2} \|\Delta_hv_h\| \int_0^\infty e^{s\cos\theta}s^{1-\alpha}ds \le  { c\tau^2} t_n^{\alpha-2} \| \Delta_hv_h \|_{L^2(\Omega)},}
\end{split}
\end{equation*}
{where the last inequality follows since for $\theta\in(\pi/2,\pi)$ and $\alpha\in(0,1)$, the integral $\int_0^\infty e^{s\cos\theta}s^{1-\alpha}ds<c$.}
Now the desired estimate follows from the triangle inequality and the
identity $\Delta_hR_h=P_h\Delta$.
\end{proof}

Next, we turn to nonsmooth initial data $v\in L^2(\Omega)$.
We begin with an estimate on the kernel.

\begin{lemma}\label{lem:est-hom-4}
Let $K_s(z)=  (z^{\al}-\Delta_h)^{-1}\Delta_h.$ By choosing
$\delta=\min(\delta_0,\delta_1)$ and $\theta={\pi}/{2}+\delta$,
there exists a $c>0$ independent of $\tau$ such that
\begin{equation*}
\| \mu(e^{-z\tau}) \beta_\tau(e^{-z\tau})^{-1} K_s( \beta_\tau(e^{-z\tau}))- z^{-1}K_s(z) \| \le   c \tau^2 |z| ,\quad \forall\, z\in \Gamma_{\theta,\delta}^\tau .
\end{equation*}
\end{lemma}
\begin{proof}
By the triangle inequality and Lemma \ref{lem:est-hom-1}, we have
\begin{equation*}
\begin{split}
&\| \mu(e^{-z\tau}) \beta_\tau(e^{-z\tau})^{-1} K_s( \beta_\tau(e^{-z\tau}))- z^{-1}K_s(z) \| \\
\le& | \mu(e^{-z\tau}) \beta_\tau(e^{-z\tau})^{-1}  - z^{-1}| \| K_s(z) \|
+|\mu(e^{-z\tau}) \beta_\tau(e^{-z\tau})^{-1}|\| K_s( \beta_\tau(e^{-z\tau}))-K_s( z)  \| \\
\le&| \mu(e^{-z\tau}) \beta_\tau(e^{-z\tau})^{-1}  - z^{-1}| \| K_s(z) \|
+{c}|z|^{-1}\| K_s( \beta_\tau(e^{-z\tau}))-K_s( z)  \| =: \mathrm{I} + {c}\mathrm{II}.
\end{split}
\end{equation*}
The first term $\mathrm{I}$ can be bounded directly using Lemma \ref{lem:est-hom-1}, \eqref{eqn:difc1}
and the inequality $   \|  K_s(z)  \| = \| I-z^\al(z^\al -\Delta_h)^{-1}   \| \le c.$
For the second term $\mathrm{II}$, it suffices to show
\begin{equation*}
  |z|\mathrm{II}=  \| K_s( \beta_\tau(e^{-z\tau}))-K_s( z)  \| \le c\tau^2|z|^2.
\end{equation*}
Using \eqref{eqn:resol}, triangle inequality, Lemma \ref{lem:est-hom-1} and \eqref{eqn:difc2}, we get
\begin{equation*}
\begin{split}
  |z|\mathrm{II}
   & = \|  \beta_\tau(e^{-z\tau})^{\al}( \beta_\tau(e^{-z\tau})^{\al} -\Delta_h)^{-1}  - z^\al(z^\al-\Delta_h)^{-1}  \| \\
   &\le | z^\al -  \beta_\tau(e^{-z\tau})^{\al}| \| ( z^{\al} -\Delta_h)^{-1} \|
   + |\beta_\tau(e^{-z\tau})|^\al \|   ( \beta_\tau(e^{-z\tau})^{\al} -\Delta_h)^{-1}-(z^\al-\Delta_h)^{-1}\| \\
   &\le c | z^\al -  \beta_\tau(e^{-z\tau})^{\al}| \| ( z^{\al} -\Delta_h)^{-1} \|
   +{c} |z|^\al \|   ( \beta_\tau(e^{-z\tau})^{\al} -\Delta_h)^{-1}-(z^\al-\Delta_h)^{-1} \| \le c\tau^2|z|^2.
\end{split}
\end{equation*}
Now the triangle inequality completes the proof of the lemma.
\end{proof}

Now we can state the temporal error for nonsmooth initial data $v\in L^2(\Omega)$.
\begin{theorem}\label{thm:v-nonsmooth}
Let $f=0$, $u_h$ and $U_h^n$ be the solutions of \eqref{eqn:fdesemidis} and \eqref{eqn:fully-mod} with
$ v \in L^2(\Omega) $, and $U_h^0= v_h=P_hv$, respectively. Then, there holds
\begin{equation*}
   \| u_h(t_n)-U_h^n \|_{L^2(\Omega)} \le c\tau^2 t_n^{-2} \| v \|_{L^2(\Omega)}, \quad n \ge 1.
\end{equation*}
\end{theorem}
\begin{proof}
We choose $\delta=\min(\delta_0,\delta_1)$ and $\theta={\pi}/{2}+\delta$.
By Proposition \ref{lem:solurep}, we split the error into
\begin{equation*}
\begin{split}
 u_h(t_n)-U_h^n &=  \frac{1}{2\pi\mathrm{i}} \int_{\Gamma_{\theta,\delta}^\tau } e^{zt_n}
  \left(  z^{-1}K_s(z)-\mu(e^{-z\tau}) \beta_\tau(e^{-z\tau})^{-1} K_s(\beta_\tau(e^{-z\tau}))\right) v_h\,dz\\
  &\quad +\frac{1}{2\pi\mathrm{i}} \int_{\Gamma_{\theta,\delta}\backslash\Gamma_{\theta,\delta}^\tau } e^{zt_n} z^{-1}K_s(z) v_h dz
  =:\mathrm{I} + \mathrm{II}.
\end{split}
\end{equation*}
By Lemma \ref{lem:est-hom-4} and choosing $\delta \le 1/t_n$, we bound the first term $\mathrm{I}$ by
\begin{equation*}
    \| \mathrm{I} \|_{L^2(\Omega)}  \le c \tau^2 \|  v_h \|_{L^2(\Omega)}\bigg(\int_{\delta}^{\pi/(\tau\sin\theta)}  e^{r t_n\cos\theta} r dr
     + \int_{-\theta}^{\theta} e^{\delta t_n|\cos\psi|} \delta^2 d\psi \bigg) \le ct_n^{ -2}\tau^2 \| v_h\|_{L^2(\Omega)}.
\end{equation*}
For the second term $\mathrm{II}$, we appeal to the resolvent estimate \eqref{eqn:resol} and obtain
\begin{equation*}
    \| \mathrm{II} \|_{L^2(\Omega)} \le c \| v_h \|_{L^2(\Omega)} \int_{\pi/(\tau\sin\theta)}^\infty  e^{r t_n\cos\theta} r^{-1} \,dr
     \le c\tau^2 t_n^{ -2} \|  v_h \|_{L^2(\Omega)}.
\end{equation*}
Now the desired result follows directly from the $L^2(\Omega)$-stability of $P_h$.
\end{proof}

\begin{remark}\label{rem:correct}
The initial correction compensates the solution singularity at $t=0$, which is
crucial to achieve the $O(\tau^2)$ convergence. Otherwise, we can only derive an $O(\tau)$ rate
\begin{equation*}
    \| u_h(t_n)-U_h^n  \|_{L^2\II} \le c \tau t_n^{\al-1} \| \Delta v \|_{L^2\II},
\end{equation*}
even if the initial data $v$ is smooth. This was numerically verified in
Table \ref{tab:1st} in Section \ref{subsec:fully}. The key of correction is to choose a proper
function $\mu$ in \eqref{eqn:fun0}, such that the estimate $|\mu(e^{-z\tau})-1| \le c \tau^2 |z|^2$
from Lemma \ref{lem:est-hom-1} holds. The choice of $\mu$ is clearly nonunique; see Section
\ref{sec:numerics} for another choice. The correction in \eqref{eqn:fully-mod} is probably
the most practical one, since it only changes the first two steps.
\end{remark}

\begin{remark}\label{rmk:exponent}{
By the proof of Theorems \ref{thm:v-smooth} and \ref{thm:v-nonsmooth} and an interpolation argument,
we deduce that for $v\in D((-\Delta)^s)$, $s\in[0,1]$, with $v_h=P_hv$, there holds
\begin{equation*}
  \|u_h(t_n)-U_h^n\|_{L^2\II} \leq c\tau^2 t_n^{s\alpha-2}\|(-\Delta_h)^sv_h\|_{L^2\II}.
\end{equation*}}
\end{remark}

\subsection{Error analysis for the inhomogeneous problem}
Now we turn to the inhomogeneous problem $f\neq0$ and $v=0$. By Proposition
\ref{lem:solurep}, it suffices to analyze the two terms involving $F_h^0=f_h(0)$ and $\widehat{G}_h$ in the integral representation.
First, assume that $f$ is time-independent. Then by \eqref{eqn:semisol}, we have
\begin{equation*}
  u_h(t_n)-U_h^n= -\frac{1}{2\pi\mathrm{i}}\int_{\Gamma_{\theta,\delta}}e^{zt_n}K(z)F_h^0dz + \frac{1}{2\pi\mathrm{i} }\int_{\Gamma_{\theta,\delta}^\tau }e^{zt_{n}}\mu(e^{-z\tau}) K( \beta_\tau(e^{-z\tau}))F_h^0  \,dz  .
\end{equation*}
Then by Lemma \ref{lem:est-hom-3} and repeating the argument in the proof of Theorem \ref{thm:v-nonsmooth}, we deduce
\begin{equation}\label{eqn:fx}
  \|u_h(t_n)-U_h^n\|_{L^2(\Omega)} \le c \tau^{2} t_n^{\al-2} \| F_h^0  \|_{L^2(\Omega)} \le  c \tau^{2} t_n^{\al-2} \|f \|_{L^2(\Omega)}.
\end{equation}
Second, with $f(0)=0$, by the Taylor expansion of integral form
\begin{equation}\label{eqn:f}
    f_h = f_h(0) + tf_h'(0)+ t*f_h''= tf_h'(0)+ t*f_h'',
\end{equation}
it suffices to bound the errors for source terms of the form $tg_h$ and $t\ast g_h$, which is done next.
The next lemma gives an error estimate for $tg_h$.
\begin{lemma}\label{lem:inh-1}
Let $v=0$, and $u_h$ and $U_h^n$ be the solutions of \eqref{eqn:fdesemidis} with $f_h = t g_h(x) \in X_h$,
and \eqref{eqn:fully-mod} with $F_h^n= f_h(t_n)$, respectively.
Then, there holds
\begin{equation*}
   \| U_h^n - u_h(t_n)\|_{L^2(\Omega)} \le c \tau^2 t_n^{\alpha-1} \| g_h \|_{L^2(\Omega)}.
\end{equation*}
\end{lemma}
\begin{proof}
Like before, we choose $\delta=\min(\delta_0,\delta_1)$ and $\theta={\pi}/{2}+\delta$.
By \eqref{eqn:semisol} and Proposition \ref{lem:solurep}, the semidiscrete Galerkin solution $u_h(t_n)$
and fully discrete solution $U_h^n$ are given by
\begin{equation*}
     u_h(t_n)  = \frac{1}{2\pi\mathrm{i}}\int_{\Gamma_{\theta,\delta}} e^{zt_{n}}z^{-2}(z^\alpha-\Delta_h)^{-1} g_h \, dz ,
     \end{equation*}
  and
  \begin{equation*}
U_h^n  = \frac{1}{2\pi\mathrm{i}}\int_{\Gamma_{\theta,\delta}^\tau } e^{zt_{n}} \frac{\tau^2 e^{-z\tau} }{ (1-e^{-z\tau})^2}(\beta_\tau(e^{-z\tau})^\alpha-\Delta_h)^{-1} g_h \, dz,
\end{equation*}
respectively. Next, we claim the following estimate on the kernels in the solution representations
\begin{equation}\label{eqn:est1}
   \| \frac{\tau^2 e^{-z\tau} }{ (1-e^{-z\tau})^2}(\beta_\tau(e^{-z\tau})^\alpha-\Delta_h)^{-1} - z^{-2}(z^\al-\Delta_h)^{-1} \| \le c \tau^2 |z|^{-\alpha}\quad z\in \Gamma_{\theta,\delta}^\tau.
\end{equation}
This is a direct consequence of Lemmas \ref{lem:est-hom-2} and \ref{lem:est-hom-1}  and the following
inequality
\begin{equation*}
\begin{split}
    |(1-e^{-z\tau})^2e^{z\tau}\tau^{-2} - z^2| 
    & \le c\tau^2|z|^4\sum_{n=1}^\infty \frac{2}{(2n+2)!} \tau^{2n-2}|z|^{2n-2}\le c\tau^2|z|^4,
    \quad\forall\, z\in \Gamma_{\theta,\delta}^\tau ,
\end{split}
\end{equation*}
where the last step holds since $|z|\tau\leq c$ for $z\in \Gamma_{\theta,\delta}^\tau$.
Next, we split the error into
\begin{equation*}
\begin{split}
 u_h(t_n)-U_h^n &=  \frac{1}{2\pi\mathrm{i}} \int_{\Gamma_{\theta,\delta}^\tau } e^{zt_n}
  \left( z^{-2}(z^\al-\Delta_h)^{-1} - \frac{\tau^2 e^{-z\tau} }{ (1-e^{-z\tau})^2}(\beta_\tau(e^{-z\tau})^\alpha-\Delta_h)^{-1} \right) g_h\,dz\\
  &\quad +\frac{1}{2\pi\mathrm{i}} \int_{\Gamma_{\theta,\delta}\backslash\Gamma_{\theta,\delta}^\tau } e^{zt_n} z^{-2}(z^\alpha-\Delta_h)^{-1} g_h dz
  :=\mathrm{I} + \mathrm{II}.
\end{split}
\end{equation*}
Using the estimate \eqref{eqn:est1} and choosing $\delta\le1/t_n$, we bound the first term $\mathrm{I}$ by
\begin{equation*}
    \|\mathrm{I}\|_{L^2(\Omega)}  \le c\tau^2 \|  g_h \|_{L^2(\Omega)}\bigg(\int_{\delta}^{\pi/(\tau\sin\theta)}  e^{r t_n\cos\theta} r^{-\alpha} dr
     + \int_{-\theta}^{\theta}  e^{\delta t_n|\cos\psi|} \delta^{1-\alpha}d\psi \bigg) \le c \tau^2t_n^{ \alpha-1} \| g_h\|_{L^2(\Omega)}.
\end{equation*}
Similarly, by appealing to the resolvent estimate \eqref{eqn:resol}, we obtain
\begin{equation*}
\begin{split}
    \| \mathrm{II} \|_{L^2(\Omega)} &\le c\| g_h \|_{L^2(\Omega)} \int_{\pi/(\tau\sin\theta)}^\infty  e^{r t_n\cos\theta} r^{-2-\alpha} \,dr
     \le  c\tau^2 t_n^{ \alpha-1} \|  g_h \|_{L^2(\Omega)}.
\end{split}
\end{equation*}
Now the desired result follows from the triangle inequality.
\end{proof}

The next lemma gives an error estimate for $t\ast g_h$.
\begin{lemma}\label{lem:inh-2}
Let $v_h=0$, $u_h$ and $U_h^n$ be the solutions of \eqref{eqn:fdesemidis} with $f_h = t * g_h \in X_h$
and \eqref{eqn:fully-mod} with $F_h^n= f_h(t_n)$, respectively. Then, there holds
\begin{equation*}
   \| U_h^n - u_h(t_n)\|_{L^2(\Omega)} \le c \tau^2 \int_0^{t_n} (t_n-s)^{\alpha-1} \| g_h(s) \|_{L^2(\Omega)} \,ds.
\end{equation*}
\end{lemma}
\begin{proof}
We choose $\delta=\min(\delta_0,\delta_1)$ and $\theta={\pi}/{2}+\delta$ like before.
First, we introduce the operator $\mathcal{E}(t)$ defined by
$\mathcal{E}  (t)   =  \frac{1}{2\pi\mathrm{i}}\int_{\Gamma_{\theta,\delta}}
 e^{z t } (z^\alpha-\Delta_h)^{-1}\, dz$.
Then, the semidiscrete Galerkin solution $u_h(t_n)$ can be represented by
\begin{equation*}
  u_h(t_n) = (\mathcal{E}*f_h)(t_n)=(\mathcal{E}*(t*g_h))(t_n)=((\mathcal{E}*t)*g_h)(t_n).
\end{equation*}
Next we derive the representation of the fully discrete solution $U_h^n$.
Using the generating function $ \widetilde f_h(\xi)=\sum_{n=0}^\infty f_h(t_n)\xi^n$,
and $ \widetilde U_h(\xi) =  ({\beta_\tau(\xi)^\alpha}-\Delta_h)^{-1} \widetilde f_h(\xi)
    =: \widetilde {\mathcal{E}} (\beta_\tau(\xi)) \widetilde f_h(\xi),
$
we represent $U_h^n$ by
\begin{equation*}
  U_h^n = \sum_{j=0}^n \mathcal{E}_\tau^{n-j}f_h(t_j) \qquad\text{with}\qquad
  \widetilde {\mathcal{E}}(\beta_\tau(\xi)) = \sum_{n=0}^\infty \mathcal{E}_\tau^n \xi^n.
\end{equation*}
Simple computation yields the following integral representation
\begin{equation*}
    \mathcal{E}_\tau^n = \frac{\tau}{2\pi\mathrm{i}}\int_{\Gamma_{\theta,\delta}^\tau } e^{zn\tau} ({ \beta_\tau(e^{-z\tau}) ^\alpha}-\Delta_h)^{-1}\,dz.
\end{equation*}
Using Lemma \ref{lem:est-hom-1}, we have the following estimate
\begin{equation}\label{eqn:Etaun}
    \|   \mathcal{E}_\tau^n  \| \le c\tau t_{n}^{\al-1}.
\end{equation}
Let $\mathcal{E}_\tau(t) = \sum_{ n=0}^\infty\mathcal{E}_\tau^{n}\delta_{t_n}(t) $,
with $\delta_{t_n}$ being the Dirac-delta function at $t_n$ (from the left side). Then we have
\begin{equation*}
  U_h^n = (\mathcal{E}_\tau*f_h)(t_n)=(\mathcal{E}_\tau*(t*g_h))(t_n)
  =((\mathcal{E}_\tau*t)*g_h)(t_n).
\end{equation*}
By the discrete convolution rule, we have
\begin{equation*}
 \widetilde{(\mathcal{E}_\tau*t)}(\xi) =  \sum_{n=0}^\infty \sum_{j=0}^{n} \mathcal{E}_\tau^{n-j} t_j \xi^n
 =  \big( \sum_{j=0}^\infty  \mathcal{E}_\tau^{j}\xi^j\big) \big( \sum_{j=0}^\infty  t_j \xi^j \big)
  = \widetilde {\mathcal{E}} (\beta_\tau(\xi))\frac{\tau \xi}{(1-\xi)^2},
\end{equation*}
and consequently, by repeating the argument in the proof of Lemma \ref{lem:inh-1}, we deduce
\begin{equation*}
 \| ((\mathcal{E}_\tau-\mathcal{E})*t)(t_n)\|  \le c\tau^2 t_n^{\alpha-1}.
\end{equation*}
Next, we derive that for $t>0$
\begin{equation}\label{eqn:tn}
 \| ((\mathcal{E}_\tau-\mathcal{E})*t)(t)  \| \le c\tau^2 t^{\alpha-1},\quad \forall t\in(t_{n-1},t_n).
\end{equation}
To see the claim, we recall the Taylor expansion of $\mathcal{E}(t)$ at $t=t_n$
\begin{equation*}
    (\mathcal{E}*t)(t) = (\mathcal{E}*t)(t_n) + (t-t_n) (\mathcal{E}*1)(t_n)
    + \int_{t_n}^t (t-s)\mathcal{E}(s)\,ds.
\end{equation*}
This expansion holds also for $(\mathcal{E}_\tau*t)(t) $. Then the preceding argument yields
\begin{equation*}
  \|  ((\mathcal{E}-\mathcal{E}_\tau)*t)(t_n)\| \le c\tau^2t_n^{\al-1}
   \quad \text{and}\quad \|(\mathcal{E}-\mathcal{E}_\tau)*1)(t_n)|| \le c\tau t_n^{\al-1}.
\end{equation*}
Meanwhile, by the resolvent estimate \eqref{eqn:resol}, we have $\|\mathcal{E}(t)\|\leq ct^{\alpha-1}$,
and consequently, there holds
\begin{equation*}
    \bigg\| \int_{t_n}^t (t-s)\mathcal{E}(s)\,ds \bigg\|
    \le c \int_t^{t_n}(s-t) s^{\al-1} \,ds \le c\tau^2 t^{\alpha-1} .
\end{equation*}
Similarly, appealing to \eqref{eqn:Etaun}, we deduce
\begin{equation*}
    \bigg\| \int_{t_n}^t (t-s)\mathcal{E}_\tau(s)\,ds\bigg\|
   \le  \tau \| \mathcal{E}_\tau^n \| \le c \tau^2 t_n^{\alpha-1}.
\end{equation*}
Then \eqref{eqn:tn} follows directly by $t_n^{\alpha-1}\le t^{\alpha-1}$ for $t\in(t_{n-1},t_n)$
and $\al\in(0,1)$, concluding the proof.
\end{proof}

By \eqref{eqn:fx} and Lemmas \ref{lem:inh-1} and \ref{lem:inh-2}, we obtain
the error estimate for the inhomogeneous problem.
\begin{theorem}\label{thm:inhomog-fully}
Let $v=0$, $f\in  W^{1, \infty} (0,T;
L^2(\Omega))$, $\int_0^t (t-s)^{\alpha-1}  \| f''(s)  \|_{L^2(\Omega)}ds\in L^\infty(0,T)$,
and $u_h$ and $U_h^n$ be the solutions of \eqref{eqn:fdesemidis} with $f_h=P_hf$
and \eqref{eqn:fully-mod} with $F_h^n=P_hf(t_n)$, respectively. Then, there holds
\begin{equation*}
\begin{split}
\| u_h(t_n)-U_h^n &\|_{L^2(\Omega)}\le c\tau^2\bigg(
t_n^{\alpha-2}\|  f(0) \|_{L^2(\Omega)}+t_n^{\alpha-1}\|  f'(0) \|_{L^2(\Omega)}+
\int_0^{t_n} (t_n-s)^{\alpha-1}\| f''(s)\|_{L^2(\Omega)} \,ds \bigg).
\end{split}
\end{equation*}
\end{theorem}

\begin{remark}
The estimate in Theorem \ref{thm:inhomog-fully} agrees with
the regularity theory for \eqref{eqn:fde}. In case $v=0$, following the splitting
\eqref{eqn:f}, one can show that the solution $u$ of problem \eqref{eqn:fde} satisfies
\begin{equation*}
\begin{split}
  \|  \partial_t^2 u(t)  \|_{L^2\II} 
  &\le c \bigg( t^{\al-2} \| f(0)  \|_{L^2\II} + t^{\al-1} \| f'(0)  \|_{L^2\II} + \int_0^t (t-s)^{\al-1}\| f''(s)  \|_{L^2(\Omega)} \,ds\bigg).
\end{split}
\end{equation*}
Hence, in order to have an $O(\tau^2)$ rate, we require $f(0)$, $f^\prime(0)\in L^2(\Omega)$
and a certain integrability of $f^{\prime\prime}(t)$. Otherwise,  the
scheme \eqref{eqn:fully-mod} might lose its second-order accuracy.
\end{remark}

\section{Numerical experiments and discussions}\label{sec:numerics}

Now we present examples on the unit square $\Omega=(0,1)^2$ to illustrate the scheme \eqref{eqn:fully-mod}.
In the computations, we divide the unit interval $(0,1)$ into $M$ equally
spaced subintervals, with a mesh size $h=1/M$, which partitions the domain $\Omega$ into $M^2$ small
squares. Then we get a symmetric mesh by connecting the diagonal of each small square.
We fix the time step size $\tau$ at $\tau=t/N$, where $t$ is the time of interest.
To examine the temporal convergence rates, we always fix the mesh size $h$ at $h=1/500$, and employ a time step
size $\tau=t/1000$ to compute the reference solution $u_h(t)$. Throughout, we measure the error $e^n= u_h(t_n)-U_h^n $ by
the normalized $L^2(\Omega)$ error $\| e^n\|_{L^2\II}/\| v \|_{L^2\II}$. Since the spatial discretization
error has been examined in \cite{JinLazarovZhou:SIAM2013,JinLazarovPasciakZhou:IMA2014}, we shall
focus on the temporal error below.

We consider the following four examples to illustrate the convergence analysis.
\begin{itemize}
  \item[$\mathrm{(a)}$]  $v=xy(1-x)(1-y)\in D(\Delta)$ and $f=0$;
  \item[$\mathrm{(b)}$] $v=\chi_{(0,1/2]\times(0,1)}(x,y) \in D((-\Delta)^{1/4-\epsilon})$ with $\epsilon\in(0,1/4)$ and $f=0$;
  \item[$\mathrm{(c)}$] $v=0$, and $f= (1+t^{1.5}) \chi_{(0,1/2]\times(0,1)}(x,y)$;
  \item[$\mathrm{(d)}$] $v=0$, and $f= t^\beta \chi_{(0,1/2]\times(0,1)}(x,y)$ with $\beta\in(0,1)$;
\end{itemize}

First we examine the convergence for the homogenous problem. The numerical results for cases (a) and (b) at
the time $t=1$ are given in Tables \ref{tab:a-error} and \ref{tab:b-error}, where \texttt{rate}
in the last column refers to the empirical convergence rate, and the numbers in the bracket denote theoretical
predictions from Section \ref{sec:conv}. It is observed that the corrected scheme
\eqref{eqn:fully-mod} exhibits a steady $O(\tau^2)$ convergence for both
smooth and nonsmooth data, which shows clearly its robustness. In Tables \ref{tab:a-error} and
\ref{tab:b-error}, we also include the numerical results by CQ generated by the second-order
BDF (SBD) and L1-2 scheme. In theory, SBD is $O(\tau^2)$ accurate for both smooth and nonsmooth problem
data \cite{JinLazarovZhou:SISC2016}, but a complete convergence analysis of the L1-2 scheme is still to be
developed, with a local truncation error $O(\tau^{3-\al})$  \cite{GaoSunZhang:2014}.
The L1-2 scheme exhibits only an $O(\tau)$ convergence, due to the insufficient solution regularity even
for smooth problem data, which contrasts sharply with the scheme \eqref{eqn:fully-mod} and SBD. Numerically,
with the same time step size $\tau$, the scheme \eqref{eqn:fully-mod} is slightly more accurate than SBD.

\begin{table}[htb!]
\caption{The $L^2$-error for Example (a) at $t=1$ with $h=1/500$.
}\label{tab:a-error}
\centering
     \begin{tabular}{|c|c|ccccc|c|}
     \hline
      Scheme &$\alpha\backslash N$  &$10$ &$20$ &$40$ & $80$ & $160$ &rate \\
     \hline
     &$0.2$   &4.87e-5 &1.14e-5   &2.76e-6 &6.73e-7 &1.62e-7 &$\approx$ 2.06 (2.00)\\
     CN &$0.5$   &1.19e-4 &2.78e-5 &6.70e-6 &1.63e-6 &3.93e-7 &$\approx$ 2.06 (2.00)\\
     &$0.8$   &1.22e-4 &2.68e-5 &6.46e-6 &1.57e-6 &3.78e-7  &$\approx$ 2.05 (2.00)\\
           \hline
         &$0.2$   &7.78e-5 &1.82e-5   &4.38e-6 &1.07e-6 &2.64e-7 &$\approx$ 2.02 (2.00)\\
     SBD   &$0.5$   &1.44e-4 &3.34e-5 &8.03e-6 &1.96e-6 &4.82e-7 &$\approx$ 2.05 (2.00)\\
         &$0.8$   &1.67e-4 &3.88e-5 &9.26e-6 &2.26e-6 &5.52e-7  &$\approx$ 2.05 (2.00)\\
           \hline
     &$0.2$   &5.02e-4 &2.49e-4   &1.23e-4 &6.00e-5 &2.87e-5 &$\approx$ 1.03 (2.80)\\
     L1-2 &$0.5$   &1.57e-3 &7.72e-4 &3.73e-4 &1.78e-4 &8.29e-5 &$\approx$ 1.06 (2.50)\\
     &$0.8$   &1.63e-3 &9.42e-4 &4.67e-4 &2.23e-4 &1.04e-4  &$\approx$ 1.00 (2.20)\\
      \hline
     \end{tabular}
\end{table}

\begin{table}[htb!]
\caption{The $L^2$-error for Example (b)
for $\alpha=0.5$ with $h=1/500$.}
\label{tab:b-error}
\centering
     \begin{tabular}{|c|c|ccccc|c|}
     \hline
     scheme &$t\backslash N$  &$10$ &$20$ &$40$ & $80$ & $160$ &rate \\
     \hline
     &$1$   &7.48e-5 &1.74e-5   &4.18e-6 &1.02e-6 &2.45e-7 &$\approx$ 2.05 (2.00)\\
     CN &$0.01$   &4.74e-4 &1.10e-4 &2.66e-5 &6.51e-6 &1.56e-6 &$\approx$ 2.05 (2.00)\\
     &$0.001$   &5.45e-4 &1.28e-4 &3.09e-5 &7.55e-6 &1.82e-6  &$\approx$ 2.05 (2.00)\\
      \hline
           &$1$   &8.98e-5 &2.08e-5   &5.00e-6 &1.22e-6 &3.00e-7 &$\approx$ 2.05 (2.00)\\
     SBD &$0.01$   &5.23e-4 &1.22e-4 &2.94e-5 &7.22e-6 &1.77e-6 &$\approx$ 2.04 (2.00)\\
     &$0.001$   &5.46e-4 &1.29e-4 &3.10e-5 &7.63e-6 &1.88e-6  &$\approx$ 2.04 (2.00)\\
      \hline
                                &$1$   &3.52e-4 &1.73e-4   &8.43e-5 &4.04e-5 &1.90e-5 &$\approx$ 1.05 (2.50)\\
     L1-2 &$0.01$   &2.29e-3 &1.14e-3 &5.53e-4 &2.66e-4 &1.25e-4 &$\approx$ 1.05 (2.50)\\
                                &$0.001$   &3.32e-3 &1.65e-3 &8.11e-4 &3.93e-4 &1.97e-4  &$\approx$ 1.04 (2.50)\\
      \hline
     \end{tabular}
\end{table}

\begin{figure}[hbt!]
\centering
\subfigure[case (a): smooth data]{
\includegraphics[trim = .1cm .1cm .1cm .1cm, clip=true,width=0.47\textwidth]{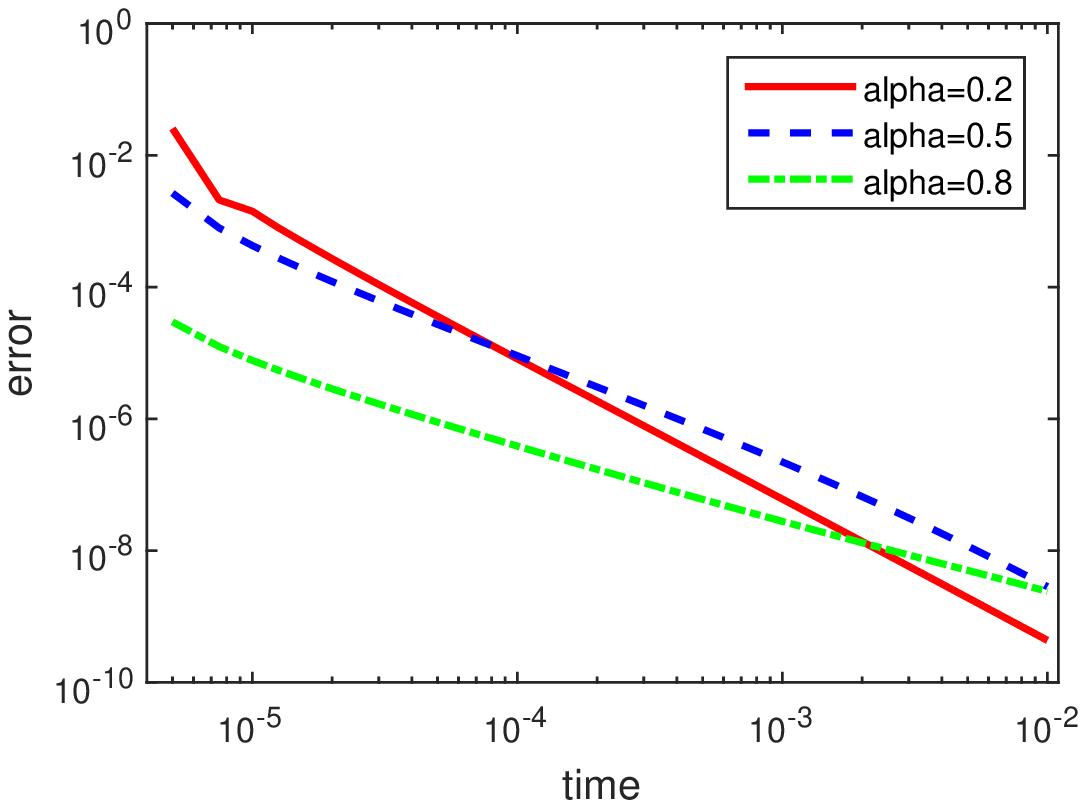}}
\subfigure[case (b): nonsmooth data]{
\includegraphics[trim = .1cm .1cm .1cm .1cm, clip=true,width=0.47\textwidth]{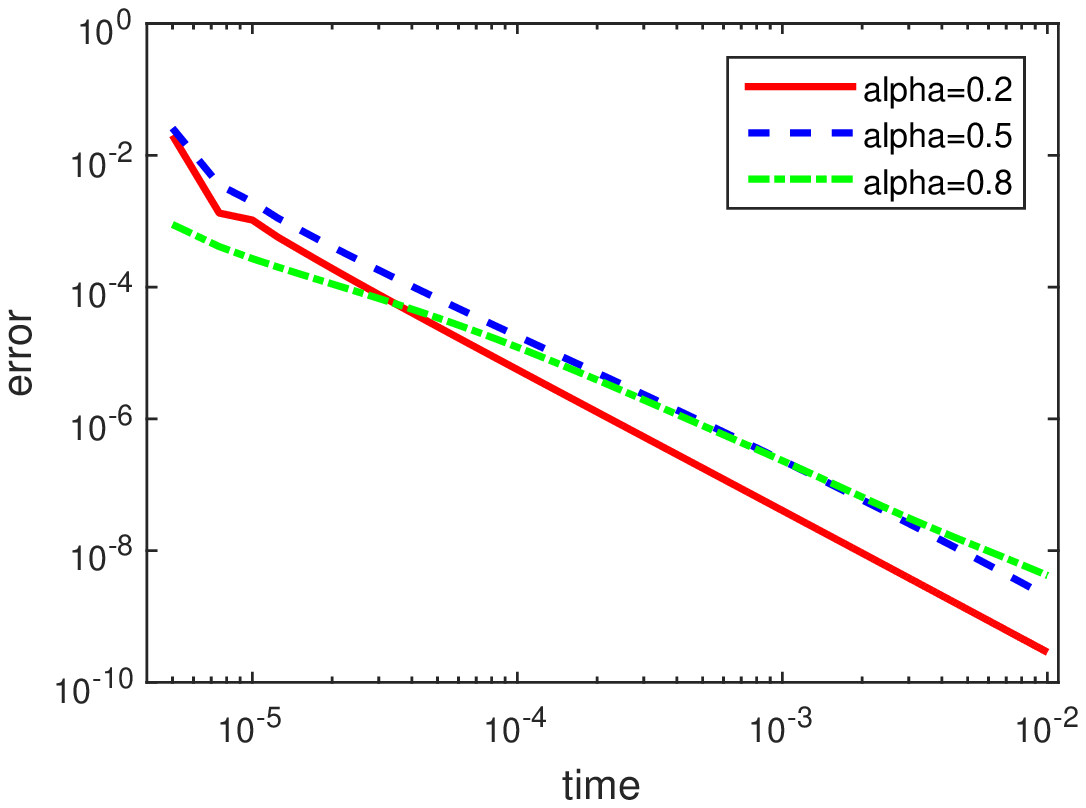}}
 \caption{The $L^2$-error versus the time $t$ with $\alpha=0.2$, $0.5$ and $0.8$, $\tau=2\times10^{-6}$, for cases (a) and (b). }\label{fig:R}
\end{figure}

\begin{table}[htb!]
\caption{The $L^2$-error for Examples (a) and (b) for $\alpha=0.5$
as $t\rightarrow0$ with $h=1/500$ and $N=10$.}
\label{tab:sing-ab}
\centering
     \begin{tabular}{|c|cccccc|c|}
     \hline
      $t$ & 1e-3 &1e-4 & 1e-5 & 1e-6 & 1e-7 & 1e-8 &rate \\
     \hline
       (a)&5.24e-4 &2.07e-4 & 6.90e-5  &2.22e-5 &7.06e-6 &2.24e-6 & 0.49 (0.50) \\
       (b) &5.43e-4 &4.06e-4 &2.97e-4 &2.23e-4 &1.67e-4 &1.24e-4& 0.13 (0.13) \\
      \hline
     \end{tabular}
\end{table}

Due to the insufficient regularity in time for problem \eqref{eqn:fde}, the temporal error
deteriorates as the time $t_n\to0$ irrespective of the data regularity; {see
Fig. \ref{fig:R} for the evolution of the $L^2$ errors with time for cases (a) and (b). For
both smooth and nonsmooth initial data, the error increases as $t$ tends to $t=0$, and
the rate is larger for smaller $\alpha$, concurring with the analysis in Section \ref{ssec:homo}.}
Next we verify the sharpness of the prefactor in the error estimates for small $t_n$. {By
Remark \ref{rmk:exponent}}, the $L^2(\Omega)$ error decays at a rate like $O(t_n^{\alpha})$
and $O(t_n^{\al/4-\al\epsilon})$ for $v\in D(\Delta)$ and $v\in D
((-\Delta)^{1/4-\epsilon})$, for any $\epsilon\in(0,1/4)$, respectively.
Hence, for a fixed $N$ and with $\alpha=1/2$, the error should behave like $O(t_n^{1/2})$ and $O(t_n^{1/8})$
for cases (a) and (b), respectively, which are fully confirmed by Table \ref{tab:sing-ab},
verifying the sharpness of the error analysis.

Next we examine the scheme \eqref{eqn:fully-mod} for inhomogeneous problems.
In case (c), $ f\in L^\infty(0,T;L^2\II)$ and $ \int_0^{t} (t-s)^{\al-1} \| f''(s)  \|_{L^2\II}\,ds
\in L^\infty(0,T)$ for any $\alpha\in(0,1)$. Theorem \ref{thm:inhomog-fully} predicts an $O(\tau^2)$ convergence,
which is fully confirmed by the results in Table \ref{tab:c-error}. The preceding observations
on the SBD and L1-2 scheme remain valid for the inhomogeneous problem: the SBD
is $O(\tau^2)$ accurate, but the L1-2 scheme can only achieve an $O(\tau)$ rate. The purpose
of case (d) is to explore the limit of the scheme \eqref{eqn:fully-mod}: In case
(d), for small exponent $\beta>0$, the source term $f$ is not smooth enough to apply
Theorem \ref{thm:inhomog-fully}, and the $O(\tau^2)$ convergence does not
hold. To see the possible convergence rate, consider the fractional
ODE $\partial_t^\alpha u(t)=t^\beta$ with  $u(0)=0$, whose solution
is given by $u(t)=\frac{\Gamma(\alpha+1)}{\Gamma(\alpha+\beta+1)}t^{\alpha+\beta}$.
The temporal regularity of the solution lies in $H^{\alpha+\beta+1/2-\epsilon}(0,T)$, from which one can
expect at best a rate $O(\tau^{\min(\alpha+\beta+1/2,2)})$. The empirical rate is of
order $O(\tau^{\min(1+\beta,2)})$ for the case $\alpha=1/2$, cf. Table \ref{tab:d-error}, which agrees well with the
expected solution regularity, thereby further verifying the robustness of the scheme \eqref{eqn:fully-mod}.

\begin{table}[htb!]
\caption{The $L^2$-error for Example (c) at $t=1$ with $h=1/500$.}
\label{tab:c-error}
\centering
     \begin{tabular}{|c|c|ccccc|c|}
     \hline
     Scheme &$\alpha\backslash N$  &$10$ &$20$ &$40$ & $80$ & $160$ &rate \\
     \hline
     &$0.2$   &1.66e-6 &3.93e-7 &9.55e-8 &2.34e-8 &5.64e-9 &$\approx$ 2.05 (2.00)\\
     CN &$0.5$   &3.52e-6 &8.23e-7 &1.99e-7 &4.86e-8 &1.17e-8 &$\approx$ 2.05 (2.00)\\
     &$0.8$   &4.49e-6 &6.80e-7 &1.63e-7 &3.97e-8 &9.54e-9 &$\approx$ 2.06 (2.00)\\
     \hline
     &$0.2$   &1.97e-6 &4.65e-7 &1.13e-7 &2.78e-8 &6.83e-9 &$\approx$ 2.04 (2.00)\\
     SBD &$0.5$   &4.32e-6 &1.00e-6 &2.42e-7 &5.91e-8 &1.45e-8 &$\approx$ 2.05 (2.00)\\
     &$0.8$   &3.49e-6 &8.03e-7 &1.90e-7 &4.61e-8 &1.13e-8 &$\approx$ 2.07 (2.00)\\
     \hline
     &$0.2$   &9.79e-6 &4.85e-6 &2.39e-6 &1.16e-6 &5.49e-7 &$\approx$ 1.04 (2.80)\\
     L1-2  &$0.5$   &1.67e-5 &8.16e-6 &3.96e-6 &1.90e-6 &8.87e-7 &$\approx$ 1.06 (2.50)\\
     &$0.8$    &1.05e-5 &4.70e-6 &2.06e-6 &8.91e-7 &3.84e-7&$\approx$ 1.19 (2.20)\\
     \hline
     \end{tabular}
\end{table}

\begin{table}[htb!]
\caption{The $L^2$-error for Example (d) at $t=1$ with $h=1/500$ and $\alpha=0.5$.
}\label{tab:d-error}
\centering
     \begin{tabular}{|c|ccccc|c|}
     \hline
      $\beta\backslash N$  &$10$ &$20$ &$40$ & $80$ & $160$ &rate \\
     \hline
     $0.2$   &9.12e-6 &3.98e-6 &1.73e-6 &7.44e-7 &3.15e-7 &$\approx$ 1.21 ($--$)\\
     $0.5$   &2.83e-6 &9.90e-7 &3.47e-7 &1.22e-7 &4.22e-8 &$\approx$ 1.52 ($--$)\\
     $0.8$   &1.07e-6 &2.97e-7 &8.28e-8 &2.31e-8 &6.43e-9 &$\approx$ 1.85 ($--$)\\
      \hline
     \end{tabular}
\end{table}

Last, we revisit the correction at starting steps. As indicated in Remark \ref{rem:correct}, there are many
possible corrections to the scheme \eqref{eqn:CN} in order to restore the
$O(\tau^2)$ accuracy. According to the convergence analysis in Section \ref{sec:conv}, the only requirement on
the correction is to choose an auxiliary function $\mu(\xi)$ in \eqref{Rep-Wh} such that
the estimate on $\mu$ in Lemma \ref{lem:est-hom-1} holds
and the resulting scheme only changes the first few steps (and thus easy to implement).
For example, the following choice satisfies the estimate
\begin{equation}\label{eqn:mu0}
    \mu_0(\xi)= (4\xi-3\xi^2+\xi^3)/[{2(1-\tfrac\al2+\tfrac\al2\xi)^{1/\al}}].
\end{equation}
Then the corresponding fully discrete scheme modifies the first three steps:
\begin{equation*}
\begin{split}
   \bPtau^\alpha (U_h-v_h)^1 - (1-\tfrac\alpha2)\Delta_h U_h^1 -  (1-\tfrac\alpha2) \Delta_h v_h&=  (1-\tfrac\alpha2) (F_h^1+ F_h^0),\\
  \bPtau^\alpha (U_h-v_h)^2  - (1-\tfrac\alpha2)\Delta_hU_h^2 - \tfrac\alpha2 \Delta_hU_h^1  -(\tfrac{3\alpha}4-\tfrac12)\Delta_h v_h
    &= (1-\tfrac\alpha2)F_h^2+\tfrac\alpha2F_h^1 + (\tfrac{3\alpha}4-\tfrac12) F_h^0,\\
  \bPtau^\alpha (U_h-v_h)^3  - (1-\tfrac\alpha2)\Delta_hU_h^3 - \tfrac\alpha2 \Delta_hU_h^2  + \tfrac\al4\Delta_h v_h
    &= (1-\tfrac\alpha2)F_h^3+\tfrac\alpha2F_h^2 -\tfrac\al4 F_h^0.
\end{split}
\end{equation*}
Then the error estimates in Section \ref{sec:conv} hold also for the correction \eqref{eqn:mu0}.
Numerically, the scheme also achieves a very steady second-order convergence, cf. Table \ref{tab:mu0-error}.

\begin{table}[htb!]
\caption{The $L^2$-error for Examples (b) and (c) at
$t=0.1$, with $h=1/500$, by the correction \eqref{eqn:mu0}.}
\label{tab:mu0-error}
\centering
     \begin{tabular}{|c|c|ccccc|c|}
     \hline
     Case & $\alpha\backslash N$  &$10$ &$20$ &$40$ & $80$ & $160$ &rate \\
     \hline
      &$0.2$   &4.61e-5 &1.08e-5   &2.61e-6 &6.37e-7 &1.53e-7 &$\approx$ 2.05 (2.00)\\
     (b) &$0.5$   &2.20e-4 &5.14e-5 &1.24e-5 &3.03e-6 &7.28e-7 &$\approx$ 2.04 (2.00)\\
      &$0.8$   &8.47e-4 &1.86e-4 &4.49e-5 &1.09e-5 &2.63e-6  &$\approx$ 2.05 (2.00)\\
      \hline
      &$0.2$   &2.16e-6 &5.09e-7   &1.22e-7 &2.92e-8 &6.97e-9 &$\approx$ 2.07 (2.00)\\
     (c) &$0.5$   &1.03e-5 &2.42e-6 &5.82e-7 &1.41e-7 &3.31e-8 &$\approx$ 2.06 (2.00)\\
      &$0.8$   &3.92e-5 &9.07e-6 &2.17e-6 &5.28e-7 &1.24e-7  &$\approx$ 2.05 (2.00)\\
      \hline
     \end{tabular}
\end{table}

\section{Conclusion}
In this work, we have analyzed a fractional Crank-Nicolson scheme for discretizing the
subdiffusion model, which naturally generalizes the classical Crank-Nicolson scheme for the heat equation to the
fractional case. We have developed essential initial corrections to robustify the scheme
which changes only the starting two steps so that
it is easy to implement and meanwhile can achieve the desired second-order
convergence for both smooth and nonsmooth problem data. A
complete convergence analysis was provided, and optimal error estimates in time directly
with respect to the data regularity were established. The accuracy, efficiency and
robustness of the corrected scheme were fully confirmed by extensive numerical experiments, and a
comparative study was included to indicate its competitiveness with existing schemes in terms of the
accuracy and efficiency.

\section*{Acknowledgements}
The research of B. Jin has been partly supported by UK EPSRC EP/M025160/1, and that of B. Li
was partially supported by the startup fund (A/C code: 1-ZE6L) provided by The Hong Kong Polytechnic University.
The research of Z. Zhou is supported in part by the AFOSR  MURI center for Material Failure Prediction through peridynamics
and the ARO MURI Grant W911NF-15-1-0562.

\bibliographystyle{abbrv}
\bibliography{frac_CN}
\end{document}